\documentclass[12pt,reqno]{amsart}
\usepackage{amscd,amsmath,amsthm,amssymb}
\usepackage{pstcol,pst-plot,pst-3d}
\usepackage{color}
\usepackage{pstricks}
\usepackage{stmaryrd}
\usepackage{tikz}
\usepackage{url}

\usepackage{amssymb}
\usepackage{amsfonts,amsmath,mathtools}
\usepackage{graphics}
\usepackage{float}
\usepackage{subfig}
\usepackage{xcolor}
\usepackage{tikz-cd}
\usepackage{booktabs} 
\usepackage{colortbl}
\usepackage[all]{xy}

\usepackage{latexsym}
\usepackage{amsfonts,amsmath,mathtools}
\usepackage{graphics}
\usepackage{float}
\usepackage{enumitem}

\newpsstyle{fatline}{linewidth=1.5pt}
\newpsstyle{fyp}{fillstyle=solid,fillcolor=verylight}
\definecolor{verylight}{gray}{0.97}
\definecolor{light}{gray}{0.9}
\definecolor{medium}{gray}{0.85}
\definecolor{dark}{gray}{0.6}

\def\NZQ{\mathbb}               

\def\QQ{{\NZQ Q}}
\def\ZZ{{\NZQ Z}}

\def\FF{{\NZQ F}}

\def\G{{\mathcal G}}

\def\HS{\textup{HS}}
\def\pd{\textup{proj}\phantom{.}\!\textup{dim}}

\def\opn#1#2{\def#1{\operatorname{#2}}} 
\opn\chara{char} \opn\length{\ell} \opn\pd{pd} \opn\rk{rk}
\opn\projdim{proj\,dim} \opn\injdim{inj\,dim} \opn\rank{rank}
\opn\depth{depth} \opn\grade{grade} \opn\height{height}
\opn\embdim{emb\,dim} \opn\codim{codim}

\opn\Tr{Tr} \opn\bigrank{big\,rank}
\opn\superheight{superheight}\opn\lcm{lcm}
\opn\trdeg{tr\,deg}
\opn\reg{reg} \opn\lreg{lreg} \opn\ini{in} \opn\lpd{lpd}
\opn\size{size} \opn\sdepth{sdepth}
\opn\link{link}\opn\fdepth{fdepth}\opn\lex{lex}
\opn\tr{tr}
\opn\type{type}
\opn\gap{gap}
\opn\diam{diam}
\opn\Mod{Mod}
\opn\div{div} \opn\Div{Div} \opn\cl{cl} \opn\Cl{Cl}
\opn\Spec{Spec} \opn\Supp{Supp} \opn\supp{supp} \opn\Sing{Sing}
\opn\Ass{Ass} \opn\Min{Min}\opn\Mon{Mon}
\opn\Ann{Ann} \opn\Rad{Rad} \opn\Soc{Soc}
\opn\Im{Im} \opn\Ker{Ker} \opn\Coker{Coker} \opn\Am{Am}
\opn\Hom{Hom} \opn\Tor{Tor} \opn\Ext{Ext} \opn\End{End}
\opn\Aut{Aut} \opn\id{id}

\opn\nat{nat}
\opn\pff{pf}
\opn\Pf{Pf} \opn\GL{GL} \opn\SL{SL} \opn\mod{mod} \opn\ord{ord}
\opn\Gin{Gin} \opn\Hilb{Hilb}\opn\sort{sort}
\opn\PF{PF}\opn\Ap{Ap}
\opn\dist{dist}
\opn\aff{aff}
\opn\relint{relint} \opn\st{st}
\opn\lk{lk} \opn\cn{cn} \opn\core{core} \opn\vol{vol}  \opn\inp{inp} \opn\nilpot{nilpot}
\opn\link{link} \opn\star{star}\opn\lex{lex}\opn\set{set}
\opn\width{wd}
\opn\Fr{F}
\opn\QF{QF}
\opn\G{G}
\opn\type{type}\opn\res{res}
\opn\conv{conv}
\opn\sr{sr}
\opn\gr{gr}

\def\pot#1#2{#1[\kern-0.28ex[#2]\kern-0.28ex]}

\opn\dirlim{\underrightarrow{\lim}}
\opn\inivlim{\underleftarrow{\lim}}
\def\Implies{\ifmmode\Longrightarrow \else
	\unskip${}\Longrightarrow{}$\ignorespaces\fi}
\def\implies{\ifmmode\Rightarrow \else
	\unskip${}\Rightarrow{}$\ignorespaces\fi}
\def\iff{\ifmmode\Longleftrightarrow \else
	\unskip${}\Longleftrightarrow{}$\ignorespaces\fi}

\let\:=\colon
\newtheorem{Theorem}{Theorem}[section]
\newtheorem{Lemma}[Theorem]{Lemma}
\newtheorem{Corollary}[Theorem]{Corollary}
\newtheorem{Proposition}[Theorem]{Proposition}
\newtheorem{Remark}[Theorem]{Remark}

\newtheorem{Example}[Theorem]{Example}

\newtheorem{Definition}[Theorem]{Definition}
\newtheorem{Setup}[Theorem]{Setup}

\newtheorem{Conjecture}[Theorem]{Conjecture}
\newtheorem{Characterization}[Theorem]{Characterization}
\newtheorem{Construction}[Theorem]{Construction}
\newtheorem{Criterion}[Theorem]{Criterion}

\let\epsilon\varepsilon
\let\kappa=\varkappa
\def\qed{\ifhmode\textqed\fi
	\ifmmode\ifinner\hfill\quad\qedsymbol\else\dispqed\fi\fi}
\def\textqed{\unskip\nobreak\penalty50
	\hskip2em\hbox{}\nobreak\hfill\qedsymbol
	\parfillskip=0pt \finalhyphendemerits=0}
\def\dispqed{\rlap{\qquad\qedsymbol}}

\opn\dis{dis}
\def\pnt{{\raise0.5mm\hbox{\large\bf.}}}

\opn\Lex{Lex}

\let\emptyset\varnothing
\let\epsilon\varepsilon

\def\alt{\textup{height}}
\def\Ass{\textup{Ass}}

\def\HS{\textup{HS}}
\def\lcm{\textup{lcm}}
\def\gcd{\textup{gcd}}
\def\set{\textup{set}}

\begin{document}
	
\title{Very well--covered graphs by Betti splittings}
\author{Marilena Crupi, Antonino Ficarra}

\address{Marilena Crupi, Department of mathematics and computer sciences, physics and earth sciences, University of Messina, Viale Ferdinando Stagno d'Alcontres 31, 98166 Messina, Italy}
\email{mcrupi@unime.it}

\address{Antonino Ficarra, Department of mathematics and computer sciences, physics and earth sciences, University of Messina, Viale Ferdinando Stagno d'Alcontres 31, 98166 Messina, Italy}
\email{antficarra@unime.it}

\subjclass[2020]{13D02, 13P10, 13F55, 13H10, 05C75}

\keywords{Complexes, Minimal resolutions, Betti splittings, Homological Shift Ideals, Vertex decomposability, Well--covered graphs, Very well--covered graphs}

\maketitle

\begin{abstract}
	A very well--covered graph is an unmixed graph without isolated vertices such that the height of its edge ideal is half of the number of vertices. We study these graphs by means of Betti splittings and mapping cone constructions.  We show that the cover ideals of Cohen--Macaulay very well--covered graphs are splittable. As a consequence, we compute explicitly the minimal graded free resolution of the cover ideals of such a class of graphs and prove that these graphs have homological linear quotients. Finally, we conjecture the same is true for each power of the cover ideal of a Cohen--Macaulay very well--covered graph, and settle it in the bipartite case.
\end{abstract}

\section*{Introduction}

Since the foundations of \textit{Algebraic Graph Theory}, one of the fundamental problem has been to classify all Cohen--Macaulay graphs \cite{RV}. Let $S = K[x_1, \ldots, x_n]$ be the standard graded polynomial ring with coefficients in a field $K$. It is well--known that if $G$ is a simple undirected graph with $\vert V(G)\vert = n$, one can associate to $G$ two squarefree monomial ideals of $S$, the \textit{edge ideal} $I(G)$ generated by all monomials $x_ix_j$ such that $x_ix_j \in E(G)$,
and the \textit{cover ideal} $J(G)$ which is generated by monomials $\prod_{x_i\in C}x_i$, for all minimal vertex covers $C$ of $G$. A \textit{minimal vertex cover} of $G$ is a subset $C$ of $V(G)$ such that each edge has at least one vertex in $C$ and no proper subset of $C$ has the same
property. $G$ is called a \textit{Cohen--Macaulay} graph if $S/I(G)$ is a Cohen--Macaulay ring. To classify all Cohen--Macaulay graphs is an hopeless task. 

However, the dual problem, namely the problem of classifying all Cohen--Macaulay cover ideals of a graph, is completely solved. Indeed the cover ideal $J(G)$ of a graph $G$ is the \textit{Alexander dual} of the edge ideal of $G$, \emph{i.e.}, $J(G) = I(G)^\vee$, and the fundamental \textit{Eagon--Reiner Criterion} says that $I(G)^\vee$ is Cohen--Macaulay if and only if $I(G)$ has a linear resolution \cite{ER98}. All graphs $G$ such that $I(G)$ has a linear resolution have been characterized by Ralph Fr\"oberg in \cite{Froberg88}. Furthermore, Herzog, Hibi and Zheng showed that for any graph $G$ such that $I(G)$ has a linear resolution, then $I(G)$ has \textit{linear powers}, \emph{i.e.}, for all $k\ge1$, $I(G)^k$ has a linear resolution \cite{HHZ2004}. 

Far beyond these classical problems, a new exciting trend has been initiated by Herzog, Moradi, Rahimbeigi and Zhu in \cite{HMRZ021a} with the introduction of the  \textit{homological shift ideals} (see, also, \cite{Bay019,BJT019,F2,F2Pack,FH2023,HMRZ021b}). Firstly, given a vector ${\bf a}=(a_1,\dots,a_n)\in\ZZ_{\ge0}^n$, the monomial $x_1^{a_1}\cdots x_n^{a_n}$ will be denoted by ${\bf x^a}$. Let $I\subset S$ be a monomial ideal and let $\FF$ be its minimal multigraded free $S$--resolution. Then, the $i$th free $S$--module in $\FF$ is $F_i=\bigoplus_{k=1}^{\beta_i(I)}S(-{\bf a}_{ik})$ with each ${\bf a}_{ik}$ an integer vector in $\ZZ^n$ with non--negative entries. 
The monomial ideal generated by the monomials ${\bf x}^{{\bf a}_{ik}}$ ($k=1, \ldots, \beta_i(I)$),  
denoted by $\HS_i(I)$, is said the \emph{$i$th homological shift ideal} of $I$. Note that $\HS_0(I)=I$. 
A basic goal of this theory is to find those combinatorial and homological properties satisfied by all $\HS_i(I)$, $i=0, \ldots, \pd(I)$.  
In such a case, we will speak about a \textit{homological shift property} of $I$, or if the context is clear, we will simply speak of \textit{homological property}. In \cite{HMRZ021a}, it is conjectured  that for any graph $G$ whose edge ideal $I(G)$ has a linear resolution or linear quotients, then $\HS_i(I(G))$ satisfy the same property 
 for all $i$, \emph{i.e.}, $I(G)$ has \textit{homological linear resolution} or \textit{homological linear quotients}. 
 The conjecture holds true for $i=1$ \cite[Theorem 4.7]{HMRZ021a}.

Inspired by the problems above, in this article we consider the dual problem of Herzog \textit{et all} in \cite{HMRZ021a}. We want to classify the graphs $G$ for which $I(G)^\vee$ has a \textit{homological linear resolution} and those with \textit{homological linear quotients}. Of course, first one needs to know all graphs $G$ such that $I(G)^\vee$ has a linear resolution. One can observe that for $I(G)^\vee$ to have a linear resolution, it is necessary that it is \textit{equigenerated}, which is equivalent to requiring that $G$ is \emph{unmixed} 
in the sense that all the minimal vertex covers of $G$ have the same cardinality. On the other hand, if $G$ is unmixed and without isolated vertices, then $2\alt(I(G)) \ge \vert V(G)\vert$ \cite{GV}. Therefore, to study the question above it is natural to restrict ourself to the ``best possible class'' of graphs: \emph{all unmixed graphs $G$ without isolated vertices and such that $2\alt(I(G)) = \vert V(G)\vert$}. In such a case, $G$ is called a \emph{very well--covered graph}. Such a class of graphs has been studied by many authors (see, for instance, \cite{CRT2011,OF,KTY2018,KPFTY,KPTY2022,Kummini2009,MMCRTY2011}). 

The outline of this article is as follows. Section \ref{sec:1} contains a quick review on Betti splittings, homological shift ideals, and some notions from the graph theory. In particular, the classification of all Cohen--Macaulay very well--covered graphs obtained by Crupi, Rinaldo and Terai in \cite{CRT2011} is fundamental for the development of the article.

In Section \ref{sec:2}, we investigate Betti splittings of the cover ideals of Cohen–Macaulay very well--covered graphs (Proposition \ref{Prop:BettiSplitVWCG}). The Betti splitting technique introduced in \cite{FHT2009} by Francisco, Ha and Van Tuyl  is a very useful tool throughout the article. 

In Section \ref{sec:3}, if $G$ is a Cohen--Macaulay very well--covered graph, we use the Betti splitting, discussed in Section \ref{sec:2}, to construct explicitly a minimal free resolution of $I(G)^\vee$. A minimal free resolution of $I(G)^\vee$ has also been constructed in \cite{KTY2018} by different tools. The main advantage of the Betti splitting technique is that one does not need to verify if the  defined complex is exact or minimal, but only to determine the suitable \textit{comparison maps} \cite[Proposition 2.1]{FHT2009}. 
The features of our resolution are required in Section \ref{sec:4} to achieve our main goal. In the last part of Section \ref{sec:3}, we draw some consequences of Theorem \ref{Thm:MinFreeResGCMverywell}. Firstly, we get a formula for the Betti numbers of $I(G)^\vee$ which is independent on the characteristic of the base field $K$ (Corollary \ref{cor:res}). Then, some formulas for the projective dimension of the cover ideal of a Cohen–Macaulay very well--covered graph $G$ are determined. In particular, in Corollary \ref{cor:pdGa(G)}, we recover a result due to Mahmoudi \textit{et all} \cite[Lemma 3.4]{MMCRTY2011} and, furthermore, we provide a new proof of it by Betti splitting (Remark \ref{rem:pdGa(G)}). Finally, we show that $I(G)^\vee$ has the \textit{alternating sum property} \cite[Definition 4.1]{Zheng2004} and also we determine a formula for the multiplicity of $S/I(G)^\vee$ (Proposition \ref{prop:summult}).

Section \ref{sec:4} contains our main result which states that the cover ideal of a Cohen--Macaulay very well--covered graph has homological linear quotients (Theorem \ref{Thm:HSCMVWCG}). Then, we infer that a very well--covered graph $G$ is Cohen--Macaulay if and only if $I(G)^\vee$ has homological linear resolution which is equivalent to  $I(G)^\vee$ having homological linear quotients (Theorem \ref{Thm:ClassificationHomLin}). 
An interesting consequence about the cover ideal of the \textit{whisker graph} $G^*$ of a graph $G$ is drawn in Corollary \ref{cor:HSwhisker}. 

Our experiments and the results obtained have lead us to conjecture that given a Cohen--Macaulay very well--covered graph $G$, all the powers of $I(G)^\vee$ have homological linear quotients (Conjecture \ref{Conj:PowersHSCMverywell}). At present we are able to prove our conjecture only for the subclass of Cohen--Macaulay bipartite graphs. For this goal, we carefully study the powers of the \textit{Hibi ideals}. In 1987, Hibi introduced a fundamental class of ideals associated to a finite poset \cite{Hibi87}. More precisely, let $(P,\succeq)$ be a finite poset, a \textit{poset ideal} $\mathcal{I}$ of $P$ is a subset of $P$ such that for any $\alpha \in\mathcal{I}$ and any $\beta \in P$ with $\beta\preceq\alpha$, then $\beta\in\mathcal{I}$. If $\mathcal{J}(P)$ is the set of all poset ideals of $P$, ordered by inclusion, then $\mathcal{J}(P)$ is a distributive lattice. Indeed, Birkhoff's fundamental theorem \cite[Theorem 3.4.1]{Stanley86} shows that any distributive lattice arises in such a way. To any $\mathcal{I}\in\mathcal{J}(P)$ one can associate the squarefree monomial $u_{\mathcal{I}}=(\prod_{p\in\mathcal{I}}x_p)(\prod_{p\in P\setminus\mathcal{I}}y_p)$ in the polynomial ring $K[\{x_p, y_p\}_{p\in P}]$. Then the Hibi ideal $H_P$ associated to $(P,\succeq)$ is the squarefree monomial ideal generated by all $u_{\mathcal{I}}$, $\mathcal{I}\in\mathcal{J}(P)$. The importance of such a class lies in the fact that the class of Hibi ideals coincides with the class of the cover ideals of Cohen--Macaulay bipartite graphs \cite{HH2005a} (see, also, \cite[Lemma 9.1.9 and Theorem 9.1.13]{JT}). Therefore, one can focus on the homological shifts of powers of Hibi ideals. In Construction \ref{Constr:NewPosetP(ell)}, we associate to any poset $(P,\succeq)$ and any integer $\ell\ge1$ a new poset $(P(\ell),\succeq_\ell)$ and then we show that the $\ell$th power of $H_P$  is equal to $H_{P(\ell)}$ up to \textit{polarization} (Theorem \ref{Thm:PosetEll}). 

Finally, observing that polarization commutes with homological shifts (Lemma \ref{Lem:HSpolarization}) and preserves also the linear quotients property (Lemma \ref{Lem:IwpSets}), we state that all the powers of an Hibi ideal have homological linear quotients (Corollary \ref{Cor:HSHPPowers}). 

All the examples in the article have been verified using \textit{Macaulay2} \cite{GDS} and the package \texttt{HomologicalShiftIdeals} \cite{F2Pack}.

\section{Preliminaries} \label{sec:1}
In this section for the reader's convenience we collect some notions and results we need for the development of the article.

\subsection{A glimpse to mapping cones and Betti splittings.} \label{sub:1}
Let $S=K[x_1,\dots,x_n]$ be a polynomial ring with coefficients in a field $K$. If $I$ is a monomial ideal of $S$,  
it is customary to ask how one can determine its minimal free resolution. 

In order to avoid any ambiguity in what follows, we assume that the trivial ideals $(0)$ and $S$ are monomial ideals since $(0)=(\emptyset)$ and $S = (1)$.

Let $G(I)$ be the unique minimal generating set of $I$. In \cite{FHT2009}, the authors pointed out that to compute a minimal graded free resolution of $I$, one can ``split" the ideal $I$ into ``smaller" ideals $P$, $Q$, \emph{i.e.}, $I=P+Q$ with $G(I)$ disjoint union of $G(P)$ and $G(Q)$.
 Hence, to get the minimal free resolution of $I$, one can use the minimal free resolutions of $P$ and $Q$ together with that of $P\cap Q$.
Indeed, let us consider the short exact sequence 
\begin{equation}\label{eq:ShortExactSequenceBettiSplitting}
0\rightarrow P\cap Q\xrightarrow{\ \psi_{-1}} P\oplus Q\xrightarrow{\ \varphi\ } I=P+Q\rightarrow0,
\end{equation}
where $\psi_{-1}(w)=(w,-w)$ and $\varphi((w,z))=w+z$. Let $\mathbb{A}$ be the minimal free resolution of $P\cap Q$ and let $\mathbb{B}$ be the minimal free resolution of $P\oplus Q$. Note that $\mathbb{B}$ is the direct sum of the minimal free resolutions of $P$ and $Q$. Since all $S$--modules $A_i$ are free, thus projective, there exists a complex homomorphism $\psi:\mathbb{A}\rightarrow\mathbb{B}$ \textit{lifting} the map $\psi_{-1}$, that means a sequence of maps $\psi_i:A_i\rightarrow B_i$ $(i\ge0)$, called the \textit{comparison maps}, making the following diagram 

\begin{equation}\label{diagram:ABpsi-1}
\begin{gathered}
\xymatrixcolsep{2.5pc}\xymatrix{
	\displaystyle
	\mathbb{A}:\cdots\ar[r] &A_2 \ar[d]_{\psi_2}\ar[r]^{d_2^A} & A_1 \ar[d]_{\psi_1}\ar[r]^{d_1^A} & A_0\ar[d]_{\psi_0}\ar[r]^-{d_0^A}&P\cap Q\ar[d]^{\psi_{-1}}\ar[r]&0\\
	\mathbb{B}:\cdots\ar[r] &B_2 \ar[r]_{d_2^B} & B_1 \ar[r]_{d_1^B} & B_0\ar[r]_-{d_0^B}&P\oplus Q\ar[r]&0
}
\end{gathered}
\end{equation}
commutative.
It is well known that $\psi$ gives rise to an acyclic complex $C(\psi)$ whose $0$th homology module is $H_0(C(\psi))=\textup{coker}(\psi_{-1})=(P\oplus Q)/\textup{Im}(\psi_{-1})\cong(P\oplus Q)/(P\cap Q)\cong P+Q=I$, \emph{i.e.}, $C(\psi)$ is a free resolution of $I$ (see, for instance, \cite[Appendix A3.12]{Ei}).\\

The complex $C(\psi)$ is defined as follows:
\begin{enumerate}
	\item[(i)] let $C_0=B_0$, and $C_i=A_{i-1}\oplus B_{i}$, for $i>0$;
	\item[(ii)] let $d_0=\varphi\circ d_0^B$, $d_1=(0, \psi_{0}+d_1^B)$, and $d_i=(-d_{i-1}^A, \psi_{i-1}+d_i^B)$, for $i>1$. 
	\end{enumerate}
This procedure, known as the \textit{mapping cone}, may be visualized as follows:
$$
\xymatrixcolsep{2.5pc}\xymatrix{
	\displaystyle
	\mathbb{A}[-1]:\cdots\ar[r] &A_2 \ar @{} [d] |{\text{\large$\oplus$}} \ar[rd]^{\psi_2}\ar[r]^{d_2^A} & A_1 \ar @{} [d] |{\text{\large$\oplus$}}  \ar[rd]^{\psi_1}\ar[r]^{d_1^A} & A_0 \ar @{} [d] |{\text{\large$\oplus$}} \ar[rd]^{\psi_0}\ar[r]^-{d_0^A}& P\cap Q\ar[r]&0&\\
	\mathbb{B}:\cdots\ar[r]&B_3\ar[r]_{d_3^B} &B_2 \ar[r]_{d_2^B} & B_1 \ar[r]_{d_1^B} &B_0\ar[r]_{d_0^B\ \ }&P\oplus Q\ar[r]&0
}
$$
Here $\mathbb{A}[-1]$ is the complex $\mathbb{A}$ \textit{homologically shifted} by $-1$.\smallskip

Unfortunately, the free resolution $C(\psi)$ is not always minimal. Below, we describe a special situation in which this happens.

For our purpose, we need to quote the next notions from \cite{FHT2009}.

\begin{Definition}\rm\label{def:Bettisplitting}
Let $I$, $P$, $Q$ be monomial ideals of $S$ such that $G(I)$ is the disjoint union of $G(P)$ and $G(Q)$. We say that $I=P+Q$ is a \textit{Betti splitting} if
\begin{equation}\label{eq:BettiSplittingI=P+Q}
\beta_{i,j}(I)=\beta_{i,j}(P)+\beta_{i,j}(Q)+\beta_{i-1,j}(P\cap Q), \ \ \ \textup{for all}\ i,j.
\end{equation}
\end{Definition}

\begin{Definition} \rm Let $I$ be a monomial ideal of $S$. Let $P$ be the ideal generated by all elements of $G(P)$ divisible by $x_i$ and let $Q$ be the ideal generated by all other elements of $G(I)$. We call $I= P+Q$ an \textit{$x_i$--partition} of $I$.\\ If $I = P+Q$ is also a Betti splitting, we call $I= P+Q$ an \textit{$x_i$--splitting}.
\end{Definition}

The following two results will be needed later.
\begin{Proposition}\label{Prop:BettiSplitxi}
	\textup{(\cite[Corollary 2.7]{FHT2009}).} Let $I=P+Q$ be a $x_i$--partition of $I$ and $P$ be the ideal generated
	by all elements of $G(I)$ divisible by $x_i$. If the minimal graded free resolution of $P$ is linear, then $I=P+Q$ is a Betti splitting.
\end{Proposition}

\begin{Proposition}\label{Prop:IBettiSplitLinRes} \textup{(\cite[Proposition 3.1]{DB}).} Let $d$ be a positive integer, $I$ a monomial ideal with a $d$--linear resolution, $P, Q\neq 0$ monomial ideals such that $I=P+Q$, $G(I) =G(P)\cup G(Q)$ and $G(P)\cap G(Q) =\emptyset$. Then the following facts are equivalent:
	\begin{enumerate} 
		\item[\em(i)] $I=P+Q$ is a Betti splitting of $I$;
		\item[\em(ii)] $P$ and $Q$ have $d$--linear resolutions.
	\end{enumerate}
	If this is the case, then $P\cap Q$ has a $(d +1)$--linear resolution.
\end{Proposition}

We close the subsection with the next result which shows that understanding when a monomial ideal has a Betti spitting is equivalent to understanding when the mapping cone (i)--(ii) gives a minimal free resolution.

\begin{Theorem}\label{Thm:MainBettiSplittingFranciscoHaTuyl}
	\textup{(\cite[Proposition 2.1]{FHT2009}).} Let $I,P,Q$ be monomial ideals of $S$ such that $I=P+Q$ and $G(I)$ is the disjoint union of $G(P)$ and $G(Q)$. Furthermore, consider the short exact sequence \textup{(\ref{eq:ShortExactSequenceBettiSplitting})}.

Then the following conditions are equivalent:
	\begin{enumerate}[label=\textup{(\alph*)}]
	\item $I=P+Q$ is a Betti splitting.
		\item For all $i$ and $j$, the map
		$$\Tor_i^S(K,P\cap Q)_j\rightarrow \Tor_i^S(K,P)_j \oplus \Tor_i^S(K,Q)_j$$
		in the long exact sequence in $\Tor$  induced from \textup{(\ref{eq:ShortExactSequenceBettiSplitting})} is the zero map.
		\item Applying the mapping cone to \textup{(\ref{eq:ShortExactSequenceBettiSplitting})} gives a minimal free resolution of $I$.

	\end{enumerate}
\end{Theorem}

\subsection{A glimpse to linear quotients, linear powers and homological shift ideals.}\label{sub:2} Let $I$ be a monomial ideal of $S=K[x_1,\dots,x_n]$, $I$ has \textit{linear quotients} if for some order $u_1>\dots>u_m$ of its minimal generating set $G(I)$, all colon ideals $(u_1,\dots,u_{\ell-1}):u_{\ell}$, $\ell=2,\dots,m$, are generated by a subset of the set of variables $\{x_1,\dots,x_n\}$. In such a case $u_1>u_2>\dots>u_m$ is called an \textit{admissible order} of $I$. By \cite[Proposition 1.2.2]{JT} each colon ideal $(u_1,\dots,u_{\ell-1}):u_{\ell}$ is generated by the monomials
$$
\frac{u_j}{\gcd(u_j,u_{\ell})}=\frac{\lcm(u_j,u_{\ell})}{u_{\ell}},\ \ \ \ j=1,\dots,\ell-1.
$$
By \cite[Lemma 8.2.3]{JT}, $u_1>\!\cdots\!>u_m$ is an admissible order of $I$ if and only if for all $j<\ell$ there exist an integer $k< \ell$ and an integer $p$ such that
$$
\frac{\lcm(u_k,u_\ell)}{u_\ell}=x_p\ \ \ \text{and}\ \ \ x_p\ \ \ \text{divides}\ \ \  \frac{\lcm(u_j,u_{\ell})}{u_{\ell}}.
$$

One knows that if $I$ is generated in the same degree and it has linear quotients, then $I$ has a linear resolution. This fact is extremely useful in combinatorial commutative algebra. Indeed, if one proves that an equigenerated monomial ideal $I$ and all its powers $I^k$ ($k\ge2$) have linear quotients, then it will follow that $I$ has \textit{linear powers}, \emph{i.e.}, $I$ and all its powers have linear resolutions.

For a monomial $u=x_1^{a_1}x_2^{a_2}\cdots x_{n}^{a_n}\in S$, the integral vector ${\bf a}=(a_1,a_2,\dots,a_n)$ is called the \textit{multidegree} of $u$. We also write $u={\bf x^a}$, in particular, for ${\bf a}={\bf 0}=(0,0,\dots,0)$, ${\bf x^0}=1$; whereas $\deg(u)=a_1+a_2+\dots+a_n$ is the \textit{degree} of $u$. It is often customary to identify the multidegree ${\bf a}$ with the monomial ${\bf x^a}$. 

We quote the next definition from \cite{HMRZ021a}.
\begin{Definition}\rm Let $I\subset S$ be a monomial ideal with minimal multigraded free resolution
$$
\FF\ \ :\ \ 0 \rightarrow F_p \rightarrow F_{p-1} \rightarrow \cdots\rightarrow F_0 \rightarrow I\rightarrow 0,
$$
where $F_i = \bigoplus_{j=1}^{\beta_i(I)}S(-{\bf a}_{i,j})$. The vectors ${\bf a}_{i,j}\in\ZZ^n$, $i\ge0$, $j=1,\dots,\beta_i(I)$, are called the $i$th \textit{multigraded shifts} of $I$. The monomial ideal
	$$
	\HS_i(I)\ =\ ({\bf x}^{{\bf a}_{i,j}}\ :\ j=1,\dots,\beta_i(I))
	$$
	is called the $i$th \textit{homological shift ideal} of $I$.
\end{Definition}

Note that $\HS_0(I)=I$ and $\HS_i(I)=(0)$ for $i<0$ or $i>\pd(I)$.\medskip

The main purpose of the theory of homological shift ideals is to understand what homological and combinatorial properties are enjoyed by all $\HS_i(I)$, $i=0, \ldots, \pd(I)$. We call any such property a \textit{homological shift property} of $I$, or if the context is clear, simply \textit{homological property}. Two important homological shift properties are the following ones: $I$ has a \textit{homological linear resolution} if $\HS_i(I)$ has a linear resolution for all $i$; 
$I$ has \textit{homological linear quotients} if $\HS_i(I)$ has linear quotients for all $i$. 

Theorem \ref{Thm:MainBettiSplittingFranciscoHaTuyl}(b) implies immediately
\begin{Proposition}\label{Prop:HSBettiSplitting}
	Let $I=P+Q$ be a Betti splitting of $I\subset S$. Then 
	$$
	\HS_k(I)=\HS_{k-1}(P\cap Q)+\HS_k(P)+\HS_{k}(Q),
	$$
	for all $k\ge0$.
\end{Proposition}

Let us recall the technique of \textit{polarization} which is an operation that transforms a monomial ideal into a squarefree monomial ideal in a larger polynomial ring.

 Let $u=x_1^{b_1}x_2^{b_2}\cdots x_n^{b_n}\in S$ be a monomial. Then, the \textit{polarization} of $u$ is the monomial
$$
u^\wp=\prod_{i=1}^n\prod_{j=1}^{b_i}x_{i,j}=\prod_{\substack{i=1,\dots,n\\ b_i>0}}x_{i,1}x_{i,2}\cdots x_{i,b_i}
$$
in the polynomial ring $K[x_{i,j}:i=1,\dots,n,j=1,\dots,b_i]$. 

For a monomial $u=x_1^{b_1}x_2^{b_2}\cdots x_n^{b_n}$ of $S$, we define the \textit{$x_i$--degree} of $u$ to be the integer $\deg_{x_i}(u)=b_i=\max\{j\ge0:x_i^j\ \textup{divides}\ u\}$. Let $I$ be a monomial ideal of $S$ and set $a_i=\max\{\deg_{x_i}(u):u\in G(I)\}$, $i=1,\dots,n$. Let
$$
R=K[x_{i,j}:i=1,\dots,n,j=1,\dots,a_i]
$$
be the polynomial ring in the variables $x_{i,j}$, $i=1,\dots,n,j=1,\dots,a_i$. The \textit{polarization} of the monomial ideal $I$ is defined to be the squarefree monomial ideal $I^\wp$ of $R$ with minimal generating set $G(I^\wp)=\{u^\wp:u\in G(I)\}$. 

The polarization commutes with the homological shift ideals as next result illustrates. 
\begin{Lemma}\label{Lem:HSpolarization}
	\textup{(\cite[Corollary 1.8]{Sbarra2001}, \cite[Proposition 1.14]{HMRZ021a}).} Let $I\subset S$ be a monomial ideal. Then $\HS_i(I^\wp)=\HS_i(I)^\wp$ for all $i\ge0$.
\end{Lemma}

\subsection{A glimpse to very well--covered graphs.}
In this article, a \textit{graph} will always mean a finite undirected graph without loops or multiple
edges. Let $G$ be a graph with \textit{vertex set} $V(G)=\{x_1,\dots,x_n\}$ and with \textit{edge set} $E(G)$. Let $x_i\in V(G)$ be a vertex of $G$. The (open) \textit{neighborhood} of $x_i$ is the set
\[N(x_i)=\{x_j\in V(G):x_ix_j\in E(G)\},\]
whereas the \textit{closed neighborhood}  of $x_i$ is the set defined as follows
\[
N[x_i]=\{x_j\in V(G):x_ix_j\in E(G)\}\cup\{x_i\}.
\]

If $W \subseteq V(G)$, we denote by $G\setminus W$ the subgraph of $G$ with the vertices of $W$ and their incident edges deleted.

A subset $W$ of $V(G)$ is called a \emph{vertex cover} if every edge of $G$ is incident with at
least one vertex in $W$. A vertex cover $W$ is called a \emph{minimal vertex cover} if there
is no proper subset of $W$ which is a vertex cover of $G$. The set of all minimal vertex cover of $G$ is denote by $\mathcal{C}(G)$.

Attached to $G$ \cite{RV} there are the following squarefree monomial ideals of the polynomial ring $S=K[x_1, \ldots, x_n]$:
\[I(G) = (x_ix_j\ :\ x_ix_j \in E(G))\]
called the \emph{edge ideal} of $G$, and 
\[I(G)^\vee = (x_{i_1}\cdots x_{i_s}\ :\ \mbox{$W = \{x_{i_1},\ldots, x_{i_s}\}$ is a minimal vertex cover of $G$})
\]
called the \emph{cover ideal} of $G$.

A graph $G$ is \emph{unmixed} or \emph{well--covered} if all the minimal vertex covers have the same cardinality. In particular, all the associated primes of $I(G)$ have the same height.  A graph $G$ is called \emph{Cohen–Macaulay} over the field $K$ if
$S/I(G)$ is a Cohen--Macaulay ring. It is clear that a Cohen--Macaulay graph is unmixed \cite[Proposition 7.2.9]{RV}.

Finally, we recall that a pairing off of all the vertices of a graph $G$ is called a \emph{perfect matching}. Thus $G$ has a perfect matching if and only if $G$ has an even number of
vertices and there is a set of independent edges covering all the vertices, where for a set of independent edges we mean a set of pairwise disjoint edges \cite{RV}.\\

The next fundamental criterion due to Eagon and Reiner will be crucial.
\begin{Criterion}\label{Crit:ER}
	\textup{(\cite{ER98}, \cite[Theorem 8.1.9]{JT}).} Let $I\subset S$ be a squarefree monomial ideal. Then $I$ is Cohen--Macaulay if and only if its Alexander dual $I^\vee$ has a linear resolution.
\end{Criterion}

Now, we analyze the class of \textit{very well--covered graphs} (see, for instance, \cite{KPFTY} and the references therein). Let $G$ be an unmixed graph without isolated vertices and let $I(G)$ be its edge ideal in $S$. In \cite[Corollary 3.4]{GV}, the authors proved the following inequality:
\[2\alt (I(G)) \ge \vert V(G)\vert.\]
\begin{Definition}\rm A graph $G$ is called \textit{very well--covered} if it is unmixed without isolated vertices and with  $2\alt(I(G)) = \vert V(G)\vert$.
\end{Definition}

By \cite[Theorem 1.2]{OF}, very well--covered graphs have always perfect matchings. Hence, for a very well--covered graph $G$  with $2n$ vertices, we may assume\medskip
\\
$(*)$ $V(G)=X \cup Y$, $X \cap Y=\emptyset$,
with $X=\{x_1,\ldots,x_n\}$ a minimal vertex cover of $G$
and $Y=\{y_1,\ldots, y_n\}$ a maximal independent set of $G$
such that  $\{x_1 y_1, \ldots, x_ny_n\} \subseteq E(G)$.\medskip

It is important to point out that when we assume the condition $(*)$ for a very well--covered
graph, we do not force any restriction on the graph. Indeed,  it is only a
relabeling of the vertices.\medskip

For the reminder of this article, set $S=K[x_1,\ldots,x_n, y_1,\ldots, y_n]$ for a field $K$. For a positive integer $n$, we denote the set $\{1,2,\dots,n\}$ by $[n]$.

\begin{Theorem}\label{Thm:CRTheight}
	\textup{(\cite[Theorem 3.6]{CRT2011}).} Let $G$ be a graph with $2n$ vertices, which are not isolated, with $\alt(I(G))=n$. We assume condition $(*)$ and also we assume that if $x_iy_j\in E(G)$ then $i\le j$. Then, the following conditions are equivalent:
	\begin{enumerate}[label=\textup{(\alph*)}]
		\item $G$ is a Cohen--Macaulay very well--covered graph.
		\item The following conditions hold:
		\begin{enumerate}[label=\textup{(\roman*)}]
			\item if $x_iy_j\in E(G)$ then $x_ix_j\notin E(G)$,
			\item if $z_ix_j,y_jx_k\in E(G)$ then $z_ix_k\in E(G)$ for any distinct $i,j,k$ and $z_i\in\{x_i,y_i\}$.
		\end{enumerate}
	\end{enumerate}
\end{Theorem}

For our convenience, we reformulate Theorem \ref{Thm:CRTheight}, as follows.

\begin{Characterization}\label{char:veryWellCGCM}
	\textup{(\cite{CRT2011}, \cite[Lemma 3.1]{MMCRTY2011}).} Let $G$ be a very well--covered graph with $2n$ vertices. Then, the following conditions are equivalent.
	\begin{enumerate}[label=\textup{(\alph*)}]
		\item $G$ is Cohen--Macaulay.
		\item There exists a relabeling of $V(G)=\{x_1,\dots,x_n,y_1,\dots,y_n\}$ such that
		\begin{enumerate}[label=\textup{(\roman*)}]
			\item $X=\{x_1,\dots,x_n\}$ is a minimal vertex cover of $G$ and $Y=\{y_1,\dots,y_n\}$ is a maximal independent set of $G$,
			\item $x_iy_i\in E(G)$ for all $i\in[n]$,
			\item if $x_iy_j\in E(G)$ then $i\le j$,
			\item if $x_iy_j\in E(G)$ then $x_ix_j\notin E(G)$,
			\item if $z_ix_j,y_jx_k\in E(G)$ then $z_ix_k\in E(G)$ for any distinct $i,j,k$ and $z_i\in\{x_i,y_i\}$.
		\end{enumerate}
	\end{enumerate}
\end{Characterization}

The next example illustrates the previous characterization.

\begin{Example} \label{expl:charc} \em
The following graph $G$ is an example of a Cohen--Macaulay very well--covered  graph with $8$ vertices.

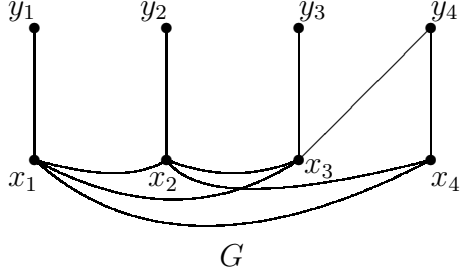
\begin{picture}(90,150)(-150,-60)
\put(-10,60){\circle*{4}}
\put(40,60){\circle*{4}}
\put(-20,65){\textit{$y_1$}}
\put(30,65){\textit{$y_2$}}
\put(-10,10){\circle*{4}}
\put(40,10){\circle*{4}}
\put(-20,0){\textit{$x_1$}}
\put(33,0){\textit{$x_2$}}
\put(90,10){\circle*{4}}
\put(92,5){\textit{$x_3$}}
\put(90,60){\circle*{4}}
\put(90,65){\textit{$y_3$}}
\put(-10,60){\line(0,-1){50}}
\put(40,60){\line(0,-1){50}}
\put(90,60){\line(0,-1){50}}
\put(90,10){\line(1,1){50}}
\put(140,0){\textit{$x_4$}}
\put(140,10){\circle*{4}}
\put(140,65){\textit{$y_4$}}
\put(140,60){\circle*{4}}
\put(140,60){\line(0,-1){50}}
\qbezier(-10,10)(25,0)(40,10)
\qbezier(-10,10)(45,-20)(90,10)
\qbezier(-10,10)(45,-40)(140,10)
\qbezier(40,10)(65,0)(90,10)
\qbezier(40,10)(55,-12)(140,10)
\put(60,-30){\textit{$G$}}
\end{picture}\vspace*{-0.9cm}

Indeed, for $S = K[x_1, x_2, x_3, x_4, y_1, y_2, y_3, y_4]$,  $\dim S/I(G) = 4 = \depth S/I(G)$ and moreover, its minimal vertex covers are: $\{x_1, x_2, x_3,x_4\}$, $\{x_1, y_2, x_3, x_4\}$, $\{y_1, x_2, x_3, x_4\}$, $\{x_1, x_2, x_3,y_4\}$, $\{x_1, x_2, y_3, y_4\}$.

The next graph $G$ is an example of a not Cohen--Macaulay very well--covered graph with $8$ vertices.\\

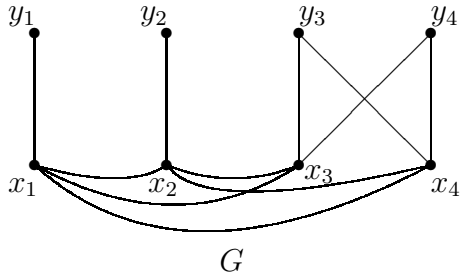
\begin{picture}(70,50)(-150,20)
\centering
\put(-10,60){\circle*{4}}
\put(40,60){\circle*{4}}
\put(-20,65){\textit{$y_1$}}
\put(30,65){\textit{$y_2$}}
\put(-10,10){\circle*{4}}
\put(40,10){\circle*{4}}
\put(-20,0){\textit{$x_1$}}
\put(33,0){\textit{$x_2$}}
\put(90,10){\circle*{4}}
\put(92,5){\textit{$x_3$}}
\put(90,60){\circle*{4}}
\put(90,65){\textit{$y_3$}}
\put(-10,60){\line(0,-1){50}}
\put(40,60){\line(0,-1){50}}
\put(90,60){\line(0,-1){50}}
\put(90,10){\line(1,1){50}}
\put(140,0){\textit{$x_4$}}
\put(140,10){\circle*{4}}
\put(140,65){\textit{$y_4$}}
\put(140,60){\circle*{4}}
\put(140,60){\line(0,-1){50}}
\put(90,60){\line(1,-1){50}}
\qbezier(-10,10)(25,0)(40,10)
\qbezier(-10,10)(45,-20)(90,10)
\qbezier(-10,10)(45,-40)(140,10)
\qbezier(40,10)(65,0)(90,10)
\qbezier(40,10)(55,-10)(140,10)
\put(60,-30){$G$}
\end{picture}
\vspace{2cm}

In fact, $G$ is very well--covered. Its minimal vertex covers are the sets $\{x_1, x_2, x_3,x_4\}$, $\{x_1, y_2, x_3, x_4\}$, $\{y_1, x_2, x_3, x_4\}$, $\{x_1, x_2, y_3,y_4\}$. But $G$ is not Cohen--Macaulay. Note that $\dim S/I(G) = 4 \neq \depth\, S/I(G)=3$, with $S = K[x_1, x_2, x_3, x_4, y_1, y_2, y_3, y_4]$.
\end{Example}

\section{Betti splittings of cover ideals of very well--covered graphs} \label{sec:2}
In this section we analyze the Betti splittings of the cover ideals of the class of Cohen--Macaulay very well--covered graphs. 

Our first result shows that if we remove some pairs of vertices of a Cohen--Macaulay very well--covered graph in a suitable way, then we obtain ``smaller graphs'' which are again Cohen--Macaulay very well--covered graphs.

\begin{Proposition}\label{Prop:GremoveVerticesVeryWC}
	Let $G$ be a Cohen--Macaulay very well--covered graph with $2n$ vertices and assume the condition $(*)$. Then $G\setminus\{x_i,y_i:i\in A\}$ is a Cohen--Macaulay very well--covered graph, for any subset $A\subseteq[n]$.
\end{Proposition}
\begin{proof} Firstly, note that since $G$ is a Cohen--Macaulay very well--covered graph with $2n$ vertices and condition $(*)$ holds, then $\alt(I(G))=n$. Moreover, from \cite[Lemma 3.5]{CRT2011}, we may assume that if $x_iy_j\in E(G)$ then $i\le j$. 
	
	If $A=\emptyset$, there is nothing to prove. Now, let $A\neq \emptyset$ and set $A=\{i_1,i_2,\dots,i_t\}$. We show that  
	 $$
	 G_0\ =\ G\setminus\{x_i,y_i:i\in A\}\ =\ G\setminus\{x_{i_1},y_{i_1},x_{i_2},y_{i_2},\dots,x_{i_t},y_{i_t}\}
	 $$
is a Cohen--Macaulay very well--covered graph. Note that $G_0$ has $2(n-t)$ vertices that are not isolated. Moreover, $G_0$ satisfies condition $(*)$ for $X_0=X\setminus\{x_{i_1},\dots,x_{i_t}\}$ and $Y_0=Y\setminus\{y_{i_1},\dots,y_{i_t}\}$, and $I(G_0)$ has height $n-t$, since $X_0$ is a minimal vertex cover of $G_0$. It is clear that $G_0$ satisfies the conditions (i)--(ii) of Theorem \ref{Thm:CRTheight}(b), because $G$ satisfies such conditions. By Theorem \ref{Thm:CRTheight}, we get that $G_0$ is also a Cohen--Macaulay very well--covered graph.
\end{proof}

Let $F\subseteq[n]$ be a non empty set, we set ${\bf x}_F=\prod_{i\in F}x_i$, ${\bf y}_F=\prod_{i\in F}y_i$. Otherwise, we set ${\bf x}_{\emptyset}={\bf y}_{\emptyset}=1$. For a monomial $u\in S= K[x_1,\dots,x_n,y_1,\dots,y_n]$, we define \textit{support} of $u$ the set 
$$
\supp(u)=\{x_i:x_i\ \textup{divides}\ u\}\cup\{y_j:y_j\ \textup{divides}\ u\}.
$$
\vspace{0,1cm}

From now, when we tell about a Cohen--Macaulay very well--covered graph $G$ with $2n$ vertices, we tacitly assume that there exists a relabeling of the set of vertices $V(G)=\{x_1,\dots,x_n,y_1,\dots,y_n\}$  which satisfy the conditions \textup{(i)--(v)} of Characterization \ref{char:veryWellCGCM}. 

\begin{Lemma}\label{Lem:minimalGeneratorsI(G)^vee}
	Let $G$ be a Cohen--Macaulay very well--covered graph with $2n$ vertices.  For each $u\in G(I(G)^\vee)$ there exists a unique subset $F$ of $[n]$ such that $u={\bf x}_F{\bf y}_{[n]\setminus F}$.
\end{Lemma}
\begin{proof}
	Let $u\in G(I(G)^\vee)$. By definition, $u$ is a squarefree monomial whose support $C=\supp(u)$ is a minimal vertex cover of $G$. Since $G$ is very well--covered, $u$ has degree $|V(G)|/2=n$. By Characterization \ref{char:veryWellCGCM}(b)(ii), $x_ny_n\in E(G)$. Hence $z_n\in C$ with $z_n\in\{y_n,x_n\}$. Note that $C_1=C\setminus\{z_n\}$ is a vertex cover of $G_1=G\setminus\{y_n,x_n\}$. By Proposition \ref{Prop:GremoveVerticesVeryWC}, $G_1$ is again Cohen--Macaulay very well--covered. Since $|C_1|=|C|-1=|V(G)|/2-1=|V(G_1)|/2$, we have that $C_1$ is a minimal vertex cover of $G_1$. Thus the monomial $u_1$ whose support is $C_1$ is a minimal generator of $I(G_1)^\vee$. By induction, $u_1={\bf x}_{F_1}{\bf y}_{[n-1]\setminus F_1}$ for a unique subset $F_1$ of $[n-1]$. If $z_n=x_n$, let $F=F_1\cup\{n\}$. Otherwise, if $z_n=y_n$, let $F=F_1$. In both cases, $u=z_nu_1={\bf x}_F{\bf y}_{[n]\setminus F}$, as required.
\end{proof}

Note that if $G$ is a Cohen--Macaulay very well--covered graph with $2n$ vertices, then 
\[
N[x_n]=\{x_{i_1},x_{i_2},\dots,x_{i_t},x_{n}, y_n\}, \,\,\, \mbox{with $i_r < n$, $r\in [t]$},
\]
and
\[N[y_n]=\{x_{j_1},x_{j_2},\dots,x_{j_p},x_n, y_n\}, \,\, \mbox{with $j_q < n$, $q\in [p]$}.
\]

Moreover, from Characterization \ref{char:veryWellCGCM}, $i_r\neq j_q$, for all $r\in [t]$ and $q\in [p]$. We will consider again such sets in the next section (see, Setup \ref{Setup:ResVeryWellCG} and Lemma \ref{Lem:N(xn)capN(y_n)}).

\begin{Proposition}\label{Prop:BettiSplitVWCG}
	Let $G$ be a Cohen--Macaulay very well--covered graph with $2n$ vertices. Let $N[x_n]=\{x_{i_1},x_{i_2},\dots,x_{i_t},x_{n}, y_n\}$, $N[y_n]=\{x_{j_1},x_{j_2},\dots,x_{j_p}, x_n, y_n\}$ and define
	\begin{align*}
	G_1\ &=\ G\setminus\{x_{i_1},y_{i_1},x_{i_2},y_{i_2},\dots,x_{i_t},y_{i_t},x_{n},y_{n}\},\\ G_2\ &=\ G\setminus\{x_{j_1},y_{j_1},\!x_{j_2},\!y_{j_2},\dots,x_{j_p},y_{j_p},x_{n},y_{n}\}.
	\end{align*}
	Then
	$$
	I(G)^\vee\ =\ x_{i_1}x_{i_2}\cdots x_{i_t}\cdot y_n I(G_1)^\vee + x_{j_1}x_{j_2}\cdots x_{j_p}\cdot x_n I(G_2)^\vee
	$$
	 is a Betti splitting.
\end{Proposition}
\begin{proof}
	Let $J_1=x_{i_1}x_{i_2}\cdots x_{i_t}\cdot y_n I(G_1)^\vee$ and $J_2=x_{j_1}x_{j_2}\cdots x_{j_t}\cdot x_n I(G_2)^\vee$. Note that $y_n$ does not divide any minimal monomial generator of $J_2$. Thus $G(J_1)\cap G(J_2)=\emptyset$. We claim that $J_1$ and $J_2$ have $n$--linear resolutions. Indeed, by Proposition \ref{Prop:GremoveVerticesVeryWC}, $G_1$ and $G_2$ are again Cohen--Macaulay very well--covered graphs and our claim follows from Criterion \ref{Crit:ER}. By virtue of Proposition \ref{Prop:BettiSplitxi}, to prove that $I(G)^\vee=J_1+J_2$ is a Betti splitting, it is enough to show that $I(G)^\vee=J_1+J_2$ is a $y_n$--partition of $I(G)^\vee$.\\
	Indeed, let $u$ be a minimal generator of $I(G)^\vee$. Then $C=\supp(u)$ is a minimal vertex cover of $G$. By Lemma \ref{Lem:minimalGeneratorsI(G)^vee}, $u={\bf x}_F{\bf y}_{[n]\setminus F}$ for some $F\subseteq[n]$. Thus, either $y_n\in C$ or $x_n\in C$. We distinguish two cases.
	\medskip\\
	\textsc{Case 1.} Suppose $y_n\in C$. Since $x_n\notin C$, but $x_{i_1}x_n,\dots,x_{i_t}x_n\in E(G)$ and $C$ is a vertex cover, then we obtain that $x_{i_1},\dots,x_{i_t}\in C$. Hence $x_{i_1}x_{i_2}\cdots x_{i_t}\cdot y_n$ divides $u$. Note that the support of $v=u/(x_{i_1}x_{i_2}\cdots x_{i_t}\cdot y_n)$ is a vertex cover of $G_1$. Furthermore, $\supp(v)$ is a minimal vertex cover, since $|\supp(v)|=n-(t+1)=|V(G_1)|/2$ and $G_1$ is a very well--covered graph (Proposition \ref{Prop:GremoveVerticesVeryWC}). Hence $v\in I(G_1)^\vee$ and $u\in J_1$ in such a case.
	\smallskip\\
	\textsc{Case 2.} Suppose $x_n\in C$. Since $y_n\notin C$, but $x_{j_1}y_n,\dots,x_{j_p}y_n\in E(G)$ and $C$ is a vertex cover, then we obtain that $x_{j_1},\dots,x_{j_p}\in C$. Hence $x_{j_1}x_{j_2}\cdots x_{j_p}\cdot x_n$ divides $u$. Note that the support of $w=u/(x_{j_1}x_{j_2}\cdots x_{j_p}\cdot x_n)$ is a vertex cover of $G_2$. Furthermore, $\supp(w)$ is a minimal vertex cover, since $|\supp(w)|=n-(p+1)=|V(G_2)|/2$ and $G_2$ is a very well--covered graph (Proposition \ref{Prop:GremoveVerticesVeryWC}). Hence, $w\in I(G_2)^\vee$ and $u\in J_2$.
	
	Hence, we have shown that $G(I(G)^\vee)$ is contained in $G(J_1)\cup G(J_2)$.\medskip
	
	For the opposite inclusion, let $v\in G(I(G_1)^\vee)$, then $\supp(v)$ is a minimal vertex cover of $G_1$. We claim that $C=\supp(x_{i_1}\cdots x_{i_t}\cdot y_nv)$ is a minimal vertex cover of $G$. Indeed, all edges of $G$ incident with a vertex belonging to $\{x_{i_1},\dots,x_{i_t},y_n\}$ 
	are incident with a vertex of 
	$C$. Since $N[x_n]\setminus\{x_n,y_n\}=\{x_{i_1},\dots,x_{i_t}\}\subset C$, then each edge which is incident with $x_n$ is also incident with a vertex of 
	$C$. Finally, let $e$ be an edge incident with $y_{i_k}$ for some $k\in[t]$. Then $e=x_jy_{i_k}\in E(G)$. Since $x_jy_{i_k}$ and $x_{i_k}x_n$ are both edges of $G$, by Characterization \ref{char:veryWellCGCM}(v), $x_jx_n\in E(G)$, too. It follows that $x_j\in N[x_n]\setminus\{x_n,y_n\}$ and $\{x_j,y_{i_k}\}\cap C\ne\emptyset$. Finally, $C$ is a minimal vertex cover of $G$. This shows that $G(J_1)\subseteq G(I(G)^\vee)$. Similarly, one proves that $G(J_2)\subseteq G(I(G)^\vee)$ by exploiting again condition (v) of Characterization \ref{char:veryWellCGCM}. The result follows.
\end{proof}
\begin{Example} \label{expl:main}
	\rm Let $S=K[x_1, \ldots, x_6, y_1, \ldots, y_6]$. Consider the graph $G$ depicted below
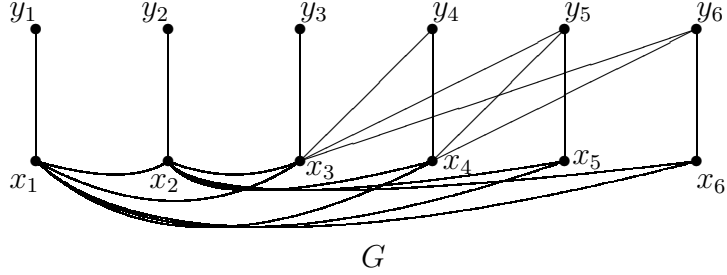
\vspace*{-0,4cm}
	
	\begin{picture}(90,150)(-90,-50)
\put(-10,60){\circle*{4}}
\put(40,60){\circle*{4}}
\put(-20,65){\textit{$y_1$}}
\put(30,65){\textit{$y_2$}}
\put(-10,10){\circle*{4}}
\put(40,10){\circle*{4}}
\put(-20,0){\textit{$x_1$}}
\put(33,0){\textit{$x_2$}}
\put(90,10){\circle*{4}}
\put(92,5){\textit{$x_3$}}
\put(90,60){\circle*{4}}
\put(90,65){\textit{$y_3$}} 
\put(-10,60){\line(0,-1){50}}
\put(40,60){\line(0,-1){50}}
\put(90,60){\line(0,-1){50}} 
\put(90,10){\line(1,1){50}}
\put(144,7){\textit{$x_4$}}
\put(140,10){\circle*{4}}
\put(140,65){\textit{$y_4$}}
\put(140,60){\circle*{4}}
\put(193,9){\textit{$x_5$}} 
\put(190,10){\circle*{4}}
\put(190,65){\textit{$y_5$}}
\put(190,60){\circle*{4}} 
\put(190,60){\line(0,-1){50}}
\put(240,0){\textit{$x_6$}} 
\put(240,10){\circle*{4}}
\put(240,65){\textit{$y_6$}}
\put(240,60){\circle*{4}} 
\put(240,60){\line(0,-1){50}}
\put(140,60){\line(0,-1){50}}
\put(190,60){\line(-2,-1){100}} 
\put(190,60){\line(-1,-1){50}} 
\put(240,60){\line(-2,-1){100}} 
\put(240,60){\line(-3,-1){150}} 
\qbezier(-10,10)(25,0)(40,10)
\qbezier(-10,10)(45,-20)(90,10)
\qbezier(-10,10)(45,-40)(140,10) 
\qbezier(-10,10)(45,-40)(190,10) 
\qbezier(-10,10)(45,-40)(240,10) 
\qbezier(40,10)(65,0)(90,10)
\qbezier(40,10)(55,-12)(140,10) 
\qbezier(40,10)(55,-12)(190,10) 
\qbezier(40,10)(55,-12)(240,10) 
\put(113,-30){\textit{$G$}}
\end{picture}\vspace*{-0,3cm}

By Characterization \ref{char:veryWellCGCM}, one verifies that $G$ is a Cohen--Macaulay very well--covered graph with $12$ vertices. We have 
\begin{align*}
I(G)\ &\ =\  (x_1y_1, x_2y_2, x_3y_3, x_4y_4, x_5y_5, x_6y_6, x_1x_2, x_1x_3, x_1x_4, x_1x_5,\\
&\phantom{\ =(..}x_1x_6, x_2x_3, x_2x_4, x_2x_5, x_2x_6, x_3y_4,x_3y_5,x_3y_6, x_4y_5, x_4y_6),\\[4pt]
I(G)^\vee\ &\ =\ (x_1x_2x_3x_4x_5x_6, y_1x_2x_3x_4x_5x_6, x_1y_2x_3x_4x_5x_6, x_1x_2x_3x_4y_5x_6,\\
&\phantom{\ =(..}x_1x_2x_3x_4x_5y_6, x_1x_2x_3x_4y_5y_6, x_1x_2x_3y_4y_5y_6, x_1x_2y_3y_4y_5y_6).
\end{align*}

Furthermore, $N[x_6]=\{x_1, x_2, x_6, y_6\}$, $N[y_6]=\{x_3, x_4, x_6, y_6\}$ and consequentaly

\[G_1 = G \setminus \{x_1, y_1, x_2, y_2, x_6, y_6\}, \,\, G_2 = G\setminus\{x_3, y_3, x_4, y_4, x_6, y_6\},\]
\emph{i.e.},

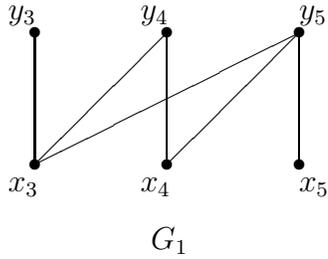
\begin{picture}(100,50)(-40,20)
\put(-10,60){\circle*{4}}
\put(40,60){\circle*{4}}
\put(-20,65){\textit{$y_3$}}
\put(30,65){\textit{$y_4$}}
\put(-10,10){\circle*{4}}
\put(40,10){\circle*{4}}
\put(-20,0){\textit{$x_3$}}
\put(30,0){\textit{$x_4$}}
\put(90,10){\circle*{4}}
\put(90,0){\textit{$x_5$}}
\put(90,60){\circle*{4}} 
\put(90,65){\textit{$y_5$}}
\put(-10,60){\line(0,-1){50}} 
\put(40,60){\line(0,-1){50}}
\put(90,60){\line(0,-1){50}}
\put(40,60){\line(-1,-1){50}} 
\put(90,60){\line(-1,-1){50}} 
\put(90,60){\line(-2,-1){100}} 
\put(34,-22){\textit{$G_1$}}
\end{picture}
\\
\\
\\
\\
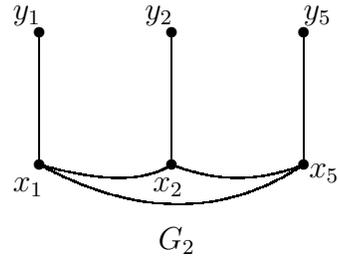
\begin{picture}(100,50)(-300,-73)
\put(-10,60){\circle*{4}}
\put(40,60){\circle*{4}}
\put(-20,65){\textit{$y_1$}}
\put(30,65){\textit{$y_2$}}
\put(-10,10){\circle*{4}}
\put(40,10){\circle*{4}}
\put(-20,0){\textit{$x_1$}}
\put(33,0){\textit{$x_2$}}
\put(90,10){\circle*{4}}
\put(92,5){\textit{$x_5$}}
\put(90,60){\circle*{4}}
\put(90,65){\textit{$y_5$}}
\put(-10,60){\line(0,-1){50}}
\put(40,60){\line(0,-1){50}} 
\put(90,60){\line(0,-1){50}} 
\qbezier(-10,10)(25,0)(40,10)
\qbezier(-10,10)(45,-20)(90,10) 
\qbezier(40,10)(65,0)(90,10) 
\put(35,-22){\textit{$G_2$}}
\end{picture}\vspace*{-1.2cm}

It follows that
\begin{eqnarray*}
I(G_1)^\vee &=& (x_3x_4x_5, x_3x_4y_5, x_3y_4y_5, y_3y_4y_5), \\
I(G_2)^\vee &= &(x_1x_2x_5, y_1x_2x_5, x_1y_2x_5, x_1x_2y_5).
\end{eqnarray*}

Finally,
 \[I(G)^\vee = x_1x_2y_6I(G_1)^\vee + x_3x_4x_6I(G_2)^\vee\]
is a Betti splitting.
\end{Example}

\section{The resolution of cover ideals of very well--covered graphs}\label{sec:3}

In this section we construct the minimal free resolution of the cover ideal $I(G)^\vee$ of any Cohen--Macaulay very well--covered graph $G$. Our method uses induction on half of the number of vertices of $G$ and the mapping cone described in Subsection \ref{sub:1}. Indeed, the mapping cone applied to the Betti splitting of $I(G)^\vee$ (Proposition \ref{Prop:BettiSplitVWCG}) yields a minimal free resolution of $I(G)^\vee$, if we know the minimal free resolutions of the cover ideals of three suitable subgraphs which we can associate to the given graph $G$. 

Firstly, let $G_1$ and $G_2$ be two graphs, we define \emph{intersection graph} of $G_1$ ad $G_2$ the graph $G$ with $V(G)=V(G_1)\cap V(G_2)$ and
$E(G)= \{e: e\in E(G_1)\cap E(G_2)\}$. We denote $G$ by $G_1\cap G_2$.

\begin{Setup}\label{Setup:ResVeryWellCG}
	\rm Let $G$ be a Cohen--Macaulay very well--covered graph with $2n$ vertices. Let $N[x_n]=\{x_{i_1},x_{i_2},\dots,x_{i_t},x_{n}, y_n\}$, $N[y_n]=\{x_{j_1},x_{j_2},\dots,x_{j_p},x_n, y_n\}$ and define
	\begin{align*}
	G_1\ &=\ G\setminus\{x_{i_1},y_{i_1},x_{i_2},y_{i_2},\dots,x_{i_t},y_{i_t},x_{n},y_{n}\},\\ G_2\ &=\ G\setminus\{x_{j_1},y_{j_1},\!x_{j_2},\!y_{j_2},\dots,x_{j_p},y_{j_p},x_{n},y_{n}\}.
	\end{align*}
	By Proposition \ref{Prop:BettiSplitVWCG}, both $G_1$ and $G_2$ are Cohen--Macaulay very well--covered graphs. Furthermore,  $\vert V(G_1)\vert = 2(n-1-t)$ and $\vert V(G_2)\vert = 2(n-1-p)$. \\
	Set $J=I(G)^\vee$, $J_1=x_{i_1}x_{i_2}\cdots x_{i_t}\cdot y_n I(G_1)^\vee$ and $J_2=x_{j_1}x_{j_2}\cdots x_{j_p}\cdot x_n I(G_2)^\vee$. Then, Proposition \ref{Prop:BettiSplitVWCG} implies that $J=J_1+J_2$ is a Betti splitting of $J$.\\	
	Finally, let us consider the subgraph $G_3=G_1\cap G_2$ of $G$. Since the structure of $G_1$ and $G_2$, it is clear that $G_3$ is Cohen--Macaulay very well--covered, too.
\end{Setup}

For instance, in Example \ref{expl:main}, $G_3 = G\setminus \{x_i, y_i: i\neq 5\}$.\\

The following lemmas will be crucial in the sequel.
\begin{Lemma}\label{Lem:N(xn)capN(y_n)}
	Assume Setup \ref{Setup:ResVeryWellCG}. Then
	$$
	N[x_n]\cap N[y_n]\ =\ \{x_n,y_n\}.
	$$
\end{Lemma}
\begin{proof}
	For all $s\in[p]$, $x_{j_s}y_n\in E(G)$. Thus Characterization \ref{char:veryWellCGCM}(iv) implies that $x_{j_s}x_n\notin E(G)$, and so $x_{j_s}\notin N[x_n]$, for all $s\in[p]$. Finally $N[x_n]\cap N[y_n]=\{x_n,y_n\}$.
\end{proof}
\begin{Lemma}\label{Lem:J12intersect}
	Assume Setup \ref{Setup:ResVeryWellCG}. Then
	$$
	 J_1\cap J_2 =x_{i_1}x_{i_2}\cdots x_{i_t}\cdot x_{j_1}x_{j_2}\cdots x_{j_p}\cdot x_n y_n I(G_3)^\vee.
	$$
\end{Lemma}
\begin{proof} Set $ J_{1,2} = J_1\cap J_2$. 
	Since $J=J_1+J_2$ is a Betti splitting and $J$ has an $n$--linear resolution,  then $J_{1,2}$ has an $(n+1)$--linear resolution (Proposition \ref{Prop:IBettiSplitLinRes}). A generating set for $J_{1,2}$ is the set $\{\lcm(u_1,u_2):u_1\in G(J_1),u_2\in G(J_2)\}$. Since $J_{1,2}$ is equigenerated in degree $n+1$, then
	$$
	G(J_{1,2})\ =\ \big\{\lcm(u_1,u_2)\ :\ u_1\in G(J_1),\ u_2\in G(J_2),\ \deg(\lcm(u_1,u_2))=n+1\big\}.
	$$
	By Lemma \ref{Lem:N(xn)capN(y_n)}, $N[x_n]\cap N[y_n]=\{x_n,y_n\}$. Thus, by the presentation of $J_{1}$ and $J_{2}$, we get that each monomial $w\in G(J_{1,2})$ is divided by $x_{i_1}\cdots x_{i_t}\cdot x_{j_1}\cdots x_{j_p}\cdot x_n y_n$. Note that $(X\cup Y)\setminus\{x_{i_1},y_{i_1},\dots,x_{i_t},y_{i_t},x_{j_1},y_{j_1},\dots,x_{j_p},y_{j_p},x_n,y_n\}$ is the vertex set of the graph $G_3$. Moreover the support of $v_1=u_1/(x_{i_1}x_{i_2}\cdots x_{i_t}\cdot x_{j_1}x_{j_2}\cdots x_{j_p}\cdot y_n)$ is a minimal vertex cover of the very well--covered graph $G_3$. Since $\lcm(u_1,u_2)=x_n u_1$, we get that $\lcm(u_1,u_2)=x_{i_1}\cdots x_{i_t}\cdot x_{j_1}\cdots x_{j_p}\cdot x_ny_n v_1\in G(x_{i_1}\cdots x_{i_t}\cdot x_{j_1}\cdots x_{j_p}\cdot x_ny_nI(G_3)^\vee)$.
	\medskip\\
	Conversely, let $w=x_{i_1}\cdots x_{i_t}\cdot x_{j_1}\cdots x_{j_p}\cdot x_ny_n v\in G(x_{i_1}\cdots x_{i_t}\cdot x_{j_1}\cdots x_{j_p}\cdot x_ny_nI(G_3)^\vee)$, with $v\in I(G_3)^\vee$. Then $\supp(x_{j_1}x_{j_2}\cdots x_{j_p}\cdot v)$ is a minimal vertex cover of $G_1$ and $\supp(x_{i_1}x_{i_2}\cdots x_{i_t}\cdot v)$ is a minimal vertex cover of $G_2$. Thus $u_1=w/x_n\in G(J_1)$ and $u_2=w/y_n\in G(J_2)$. Since $w=\lcm(u_1,u_2)$ and $\deg(w)=n+1$, we have that $w\in G(J_{1,2})$, showing the other inclusion.
\end{proof}

We now turn to the construction of the minimal free resolution. For a subset $C$ of the set of variables $X\cup Y=\{x_1,\dots,x_n,y_1,\dots,y_n\}$, we set
\begin{equation}\label{eq:notation}
{\bf z}_C= {\bf x}_{C_x}{\bf y}_{C_y},
\end{equation}
with $C_x=\{i:x_i\in C\}$ and $C_y=\{j:y_j\in C\}$. 

If $G$ is graph such that $V(G)=\{x_1,\dots,x_n,y_1,\dots,y_n\}$ and $C\in\mathcal{C}(G)$,  we define the set $\mathcal{C}(G;C)=\{x_s:y_s\in C\ \text{and}\ (C\setminus y_s)\cup x_s\in\mathcal{C}(G)\}$. Recall that $\mathcal{C}(G)$ is the set of all minimal vertex cover of $G$. \\

For instance, in Example \ref{expl:charc}, for $C = \{x_1, x_2, y_3, y_4\}$, $\mathcal{C}(G;C)=\{x_3\}$. Indeed, $\{x_1, x_2, y_4\}\cup x_3= \{x_1, x_2, x_3, y_4\}$ is a minimal vertex cover of the given graph $G$. One can observe that $x_4\notin \mathcal{C}(G;C)$, since $\{x_1, x_2, y_3\}\cup x_4$ is not a vertex cover of $G$. Moreover, for $C = \{x_1, x_2, x_3, x_4\}$, $\mathcal{C}(G;C)=\emptyset$.\medskip

In what follows, we denote by $\binom{\mathcal{C}(G;C)}{i}$ the set of all subsets of size $i$ of $\mathcal{C}(G;C)$, $0\le i\le \vert \mathcal{C}(G;C)\vert$. With abuse of notation, for $\sigma\in\binom{\mathcal{C}(G;C)}{i}$, we set ${\bf x}_\sigma= \prod_{x_s\in \sigma}x_s$. In particular, ${\bf x}_\emptyset=1$. 
\begin{Construction}\label{Constr:MinimalFreeResVeryWellCG}
	\rm Assume Setup \ref{Setup:ResVeryWellCG}. Let
	$$
	\FF\ :\ \ \dots\rightarrow F_i\xrightarrow{\ d_{i}\ }F_{i-1}\xrightarrow{\ d_{i-1}}\cdots\xrightarrow{\ d_{2}\ }F_1\xrightarrow{\ d_1\ }F_0\xrightarrow{\ d_0\ }J\rightarrow0
	$$
	be the complex
	\begin{enumerate}
		\item[-] whose $i$th free module $F_i$ has as a basis the symbols ${\bf f}(C;\sigma)$ having multidegree ${\bf z}_C{\bf x}_\sigma$, where $C\in\mathcal{C}(G)$ and $\sigma\in\binom{\mathcal{C}(G;C)}{i}$, \emph{i.e.}, $\sigma\subseteq\mathcal{C}(G;C)$ is a subset of size $i$;
		\item[-] and whose $i$th differential is given by $d_0({\bf f}(C;\emptyset))={\bf z}_C$ for $i=0$ and for $i>0$ is defined as follows:
	\end{enumerate}
	$$
	d_i({\bf f}(C;\sigma))\ =\ \sum_{x_s\in\sigma}(-1)^{\alpha(\sigma;x_s)}\big[y_s{\bf f}((C\setminus y_s)\cup x_s;\sigma\setminus x_s)-x_s{\bf f}(C;\sigma\setminus x_s)\big],
	$$
	where $\alpha(\sigma;x_s)=|\{x_j\in\sigma:j>s\}|$.
\end{Construction}

From now on, we set $J_{1,2}=J_1\cap J_2$ and denote by $(\FF_{J_{1,2}},d^{J_{1,2}})$, $(\FF_{J_1},d^{J_1})$, $(\FF_{J_2},d^{J_2})$ the minimal free resolutions of $J_{1,2}$, $J_1$, $J_2$, respectively.

\begin{Theorem}\label{Thm:MinFreeResGCMverywell}
	The complex $\FF$ given in Construction \ref{Constr:MinimalFreeResVeryWellCG} is the minimal free resolution of $J=I(G)^\vee$.
\end{Theorem}
\begin{proof}
	We proceed by strong induction on $|V(G)|/2=n$. For $n=1$, $I(G)=(x_1y_1)$ and $J=I(G)^\vee=(x_1,y_1)$. In this case one readily verifies that the complex $\FF$ of Construction \ref{Constr:MinimalFreeResVeryWellCG} is the minimal free resolution of $J=I(G)^\vee$. So, let $n>1$. The graphs $G_1,G_2,G_3$ have vertex sets whose cardinality is less than $\vert V(G)\vert$, thus by induction we can assume that the minimal free resolutions of $I(G_1)^\vee,I(G_2)^\vee,I(G_3)^\vee$ are as given in Construction \ref{Constr:MinimalFreeResVeryWellCG}. As a consequence, we know explicitly the resolutions $\FF_{J_{1,2}},\FF_{J_1},\FF_{J_2}$, since each of these resolutions is equal to one of the three previously mentioned resolutions up to multiplication by a suitable monomial.
	
	Let $(\FF_{J},d^{J})$ be the resolution obtained by the mapping cone applied to the Betti splitting $J=J_1+J_2$. Then $\FF_{J}$ is the minimal free resolution of $J$. We show that $\FF_{J}$ can be identified with $\FF$. We achieve this goal in three steps.
	\medskip\\
	\textsc{Step 1.} Let us show that the free modules of $\FF_{J}$ have the basis described in Construction \ref{Constr:MinimalFreeResVeryWellCG}. By the mapping cone (Subsection \ref{sub:1}, (i)), 
	$F_i^{J}=F_{i-1}^{J_{1,2}}\oplus F_i^{J_1}\oplus F_i^{J_2}$. 
	
	The bases of $F_i^{J_1},F_i^{J_2}$ and $F_{i}^{J_{1,2}}$ are the following ones, respectively:
	\begin{align*}
	\mathcal{B}_1\ &=\ \Big\{x_{i_1}x_{i_2}\cdots x_{i_t}\cdot y_n\cdot{\bf f}_{J_1}(C_1;\sigma_1)\ :\ C_1\in\mathcal{C}(G_1),\ \sigma_1\in\binom{\mathcal{C}(G_1;C_1)}{i}\Big\},\\
	\mathcal{B}_2\ &=\ \Big\{x_{j_1}x_{j_2}\cdots x_{j_p}\cdot x_n\cdot{\bf f}_{J_2}(C_2;\sigma_2)\ :\ C_2\in\mathcal{C}(G_2),\ \sigma_2\in\binom{\mathcal{C}(G_2;C_2)}{i}\Big\},\\
	\mathcal{B}_3\ &=\ \Big\{x_{i_1}\cdots x_{i_t}\cdot x_{j_1} \cdots x_{j_p}\cdot x_ny_n\cdot{\bf f}_{J_{1,2}}(C_3;\sigma_3)\ :\ C_3\in\mathcal{C}(G_3),\ \sigma_3\in\binom{\mathcal{C}(G_3;C_3)}{i}\Big\}.
	\end{align*}
	
	Let ${\bf f}(C;\sigma)$ be a basis element of $F_i$ as in Construction \ref{Constr:MinimalFreeResVeryWellCG}, \emph{i.e.}, with multidegree ${\bf z}_C{\bf x}_{\sigma}$ and with $C\in\mathcal{C}(G)$, $\sigma\in\binom{\mathcal{C}(G;C)}{i}$. We distinguish three possible cases.
	\smallskip\\
	\textsc{Case 1.1.} Let $x_n\in C$. Then $N[y_n]\setminus N[x_n]=\{x_{j_1},\dots,x_{j_p}\}\subset C$. Indeed, $C$ is a vertex cover of $G$, $y_n\notin C$ (Lemma \ref{Lem:minimalGeneratorsI(G)^vee}) but $x_{j_r}y_n\in E(G)$ for $r\in[p]$. Hence, $C_2=C\setminus\{x_{j_1},\dots,x_{j_p},x_n\}$ is a minimal vertex cover of $G_2$. Furthermore, for all $x_s\in\sigma$, $(C_2\setminus y_s)\cup x_s$ is a minimal vertex cover of $G_2$ because $(C\setminus y_s)\cup x_s$ is a minimal vertex cover of $G$. Since ${\bf b}=x_{j_1}x_{j_2}\cdots x_{j_p}\cdot x_n\cdot{\bf f}_{J_2}(C_2;\sigma)\in\mathcal{B}_2$ has the same multidegree of ${\bf f}(C;\sigma)$, we can identify ${\bf f}(C;\sigma)$ with ${\bf b}$.
	\smallskip\\
	\textsc{Case 1.2.} Let $y_n\in C$ and $x_n\notin\sigma$. Then $C_1=C\setminus\{x_1,\dots,x_{i_t},y_n\}$ is a minimal vertex cover of $G_1$. As before, we can identify ${\bf f}(C;\sigma)$ with $x_{i_1}\dots x_{i_t}\cdot y_n\cdot {\bf f}_{J_1}(C_1;\sigma)\in\mathcal{B}_1$.
	\smallskip\\
	\textsc{Case 1.3.} Let $y_n\in C$ and $x_n\in\sigma$. Then both $C$ and $(C\setminus y_n)\cup x_n$ are minimal vertex covers of $G$. Hence, $(N[x_n]\cup N[y_n])\setminus\{x_n,y_n\}=\{x_{i_1},\dots,x_{i_t},x_{j_1},\dots,x_{j_p}\}\subset C$.
	Setting $C_3=C\setminus\{x_{i_1},\dots,x_{i_t},x_{j_1},\dots,x_{j_p},y_n\}$ and $\sigma_3=\sigma\setminus x_n$, we obtain that $C_3$ and $(C_3\setminus y_s)\cup x_s$ are both minimal vertex covers of $G_3$ for all $x_s\in\sigma_3$. Hence, in this case we can identify ${\bf f}(C;\sigma)$ with $x_{i_1}\cdots x_{i_t}\cdot x_{j_1} \cdots x_{j_p}\cdot x_ny_n\cdot{\bf f}_{J_{1,2}}(C_3;\sigma_3)\in\mathcal{B}_3$.
	\smallskip
	
	Conversely, any basis element ${\bf b}\in\mathcal{B}_{i}$ ($i=1,2,3$) can be identified with a basis element ${\bf f}(C;\sigma)$, as given in Construction \ref{Constr:MinimalFreeResVeryWellCG}. Thus, we realize that the modules $F_i^J$ have the required bases as described in Construction \ref{Constr:MinimalFreeResVeryWellCG}.
	\medskip\\
	\textsc{Step 2.} Let $\psi_{-1}:u\in J_{1,2}\mapsto(u,-u)\in J_1\oplus J_2$. In order to apply the mapping cone, we need to construct the comparison maps $\psi_i$ making the following diagram 
	$$
	\xymatrix{
		\displaystyle
		\FF_{J_{1,2}}:\cdots\ar[r] &F_2^{J_{1,2}} \ar[d]_{\psi_2}\ar[r]^{d_2^{J_{1,2}}} & F_1^{J_{1,2}} \ar[d]_{\psi_1}\ar[r]^{d_1^{J_{1,2}}} & F_0^{J_{1,2}}\ar[d]_{\psi_0}\ar[r]^{d_0^{J_{1,2}}}& J_{1,2}\ar[d]^{\psi_{-1}}\ar[r]&0\\
		\FF_{J_{1}}\oplus\FF_{J_{2}}:\cdots\ar[r] &F_2^{J_{1}}\oplus F_2^{J_{2}} \ar[r]_{d_2^{J_{1}}\oplus d_2^{J_{2}}} & F_1^{J_{1}}\oplus F_1^{J_{2}} \ar[r]_{d_1^{J_{1}}\oplus d_1^{J_{2}}} & F_0^{J_{1}}\oplus F_0^{J_{2}}\ar[r]_{d_0^{J_{1}}\oplus d_0^{J_{2}}}&J_1\oplus J_2\ar[r]&0
	}
	$$
	commutative. Using the notation in (\ref{eq:notation}), let ${\bf b}=x_{i_1}\cdots x_{i_t}\cdot x_{j_1} \cdots x_{j_p}\cdot x_ny_n\cdot{\bf f}_{J_{1,2}}(C_3;\sigma_3)={\bf z}_{N[x_n]\cup N[y_n]}\cdot{\bf f}_{J_{1,2}}(C_3;\sigma_3)\in\mathcal{B}_3$ be a basis element of $F_i^{J_{1,2}}$. We define
	$$
	\psi_i({\bf b})\ =\ \big({\bf z}_{N[x_n]}\cdot{\bf f}_{J_1}(C_3\cup\{x_{j_1},\dots,x_{j_p}\};\sigma_3),\ -{\bf z}_{N[y_n]}\cdot{\bf f}_{J_2}(C_3\cup\{x_{i_1},\dots,x_{i_t}\};\sigma_3)\big).
	$$
	To simplify the notation, we set $C_1=C_3\cup\{x_{j_1},\dots,x_{j_p}\}$ and $C_2=C_3\cup\{x_{i_1},\dots,x_{i_t}\}$. One can note that $\psi_i$ is well defined. Indeed, $C_1\in\mathcal{C}(G_1)$, $C_2\in\mathcal{C}(G_2)$, $\sigma_3\subseteq\mathcal{C}(G_1;C_1)$, and $\sigma_3\subseteq\mathcal{C}(G_2;C_2)$.
	
	We need to verify that for all $i\ge0$ and all ${\bf b}\in\mathcal{B}_3$, it is 
	$$
	\psi_{i-1}\circ d_i^{J_{1,2}}({\bf b})\ =\ (d_i^{J_1}\oplus d_i^{J_2})\circ \psi_i({\bf b}).
	$$
	For $i=0$, we have $\sigma_3=\emptyset$, thus ${\bf x}_{\sigma_3}=1$ and both sides of the equation are equal to $$({\bf z}_{N[x_n]\cup N[y_n]}{\bf z}_{C_3},-{\bf z}_{N[x_n]\cup N[y_n]}{\bf z}_{C_3}).$$
	
	Let $i>0$. To further simplify the notation, let us write the generic basis element $({\bf z}_{N[x_n]}{\bf f}_{J_1}(C_1;\sigma_1),{\bf z}_{N[y_n]}{\bf f}_{J_2}(C_2;\sigma_2))\in F_i^{J_1}\oplus F_i^{J_2}$ as ${\bf z}_{N[x_n]}{\bf f}_{J_1}(C_1;\sigma_1)+{\bf z}_{N[y_n]}{\bf f}_{J_2}(C_2;\sigma_2)$. By induction we know explicitly $d_i^{J_{1,2}},d_i^{J_1},d_i^{J_{2}}$. Let us compute $\psi_{i-1}\circ d_i^{J_{1,2}}({\bf b})$. We have
	\small\begin{align*}
	d_i^{J_{1,2}}({\bf b})&={\bf z}_{N[x_n]\cup N[y_n]}\sum_{x_s\in\sigma_3}(-1)^{\alpha(\sigma_3;x_s)}\big[y_s{\bf f}_{J_{1,2}}((C_3\setminus y_s)\cup x_s;\sigma_3\setminus x_s)-x_s{\bf f}_{J_{1,2}}(C_3;\sigma_3\setminus x_s)\big].
	\end{align*}\normalsize
	Hence,
	\begin{equation}\label{eq:psii-1di}
	\begin{aligned}
	\psi_{i-1}\circ d_i^{J_{1,2}}({\bf b})\ =\ \sum_{x_s\in\sigma_3}(-1)^{\alpha(\sigma_3;x_s)}\big[&\ y_s{\bf z}_{N[x_n]}\cdot{\bf f}_{J_1}((C_1\setminus y_s)\cup x_s;\sigma_3\setminus x_s)\\[-1.0em]
	-&\ y_s{\bf z}_{N[y_n]}\cdot{\bf f}_{J_2}((C_2\setminus y_s)\cup x_s;\sigma_3\setminus x_s)\\
	-&\ x_s{\bf z}_{N[x_n]}\cdot{\bf f}_{J_1}(C_1;\sigma_3\setminus x_s)\\
	+&\ x_s{\bf z}_{N[y_n]}\cdot{\bf f}_{J_2}(C_2;\sigma_3\setminus x_s)
	\big].
	\end{aligned}
	\end{equation}
	Now, let us compute $(d_i^{J_1}\oplus d_i^{J_2})\circ\psi_i({\bf b})$. Firstly, we have
	$$
	\psi_i({\bf b})\ =\ {\bf z}_{N[x_n]}\cdot{\bf f}_{J_1}(C_1;\sigma_3)-{\bf z}_{N[y_n]}\cdot{\bf f}_{J_2}(C_2;\sigma_3),
	$$
	then
	\begin{align*}
	(d_i^{J_1}\oplus d_i^{J_2})&\circ\psi_i({\bf b})\ ={\bf z}_{N[x_n]}\cdot d_i^{J_1}({\bf f}_{J_1}(C_1;\sigma_3))-{\bf z}_{N[y_n]}\cdot d_i^{J_2}({\bf f}_{J_2}(C_2;\sigma_3))\\
	=&\ {\bf z}_{N[x_n]}\sum_{x_s\in\sigma_3}(-1)^{\alpha(\sigma_3;x_s)}\big[y_s{\bf f}_{J_{1}}((C_1\setminus y_s)\cup x_s;\sigma_3\setminus x_s)-x_s{\bf f}_{J_{1}}(C_1;\sigma_3\setminus x_s)\big]\\
	-&\ {\bf z}_{N[y_n]}\sum_{x_s\in\sigma_3}(-1)^{\alpha(\sigma_3;x_s)}\big[y_s{\bf f}_{J_{2}}((C_2\setminus y_s)\cup x_s;\sigma_3\setminus x_s)-x_s{\bf f}_{J_{2}}(C_2;\sigma_3\setminus x_s)\big].
	\end{align*}
	A comparison with equation (\ref{eq:psii-1di}) shows that $\psi_{i-1}\circ d_i^{J_{1,2}}({\bf b})\ =\ (d_i^{J_1}\oplus d_i^{J_2})\circ \psi_i({\bf b})$.
	\medskip\\
	\textsc{Step 3.} It remains to prove that the differentials $d_i^J$ act as described in Construction \ref{Constr:MinimalFreeResVeryWellCG}. So, let ${\bf f}(C;\sigma)$, $C\in\mathcal{C}(G)$ and $\sigma\in\binom{\mathcal{C}(G;C)}{i}$ be a basis element. Let us compute $d_i^J({\bf f}(C;\sigma))$. For $i=0$, we easily see that $d_0^J({\bf f}(C;\emptyset))={\bf z}_C$, as desired.
	\smallskip
	
	Let $i>0$. We distinguish three cases.
	\smallskip\\
	\textsc{Case 3.1.} Let  $x_n\in C$. Set $C_2=C\setminus\{x_{j_1},\dots,x_{j_p},x_n\}$. Then, by the mapping cone (Subsection \ref{sub:1}, (ii)), we have:
	\begin{align*}
	d_i^J({\bf f}&(C;\sigma))\ =\ d_i^{J_2}(x_{j_1}\cdots x_{j_p}\cdot x_n\cdot{\bf f}_{J_2}(C_2;\sigma))\\
	&=\ x_{j_1}\cdots x_{j_p}\cdot x_n\sum_{x_s\in\sigma}(-1)^{\alpha(\sigma;x_s)}\big[y_s{\bf f}_{J_2}((C_2\setminus y_s)\cup x_s;\sigma\setminus x_s)-x_s{\bf f}_{J_{2}}(C_2;\sigma\setminus x_s)\big]\\
	&=\ \sum_{x_s\in\sigma}(-1)^{\alpha(\sigma;x_s)}\big[y_s{\bf f}((C\setminus y_s)\cup x_s;\sigma\setminus x_s)-x_s{\bf f}(C;\sigma\setminus x_s)\big],
	\end{align*}
	as required.
	\smallskip\\
	\textsc{Case 3.2.} Let $y_n\in C$ and $x_n\notin\sigma$. Setting $C_1=C\setminus\{x_{i_1},\dots,x_{i_t},x_n\}$, by the mapping cone we see that $d_i^{J}({\bf f}(C;\sigma))=d_i^{J_1}(x_{i_1}x_{i_2}\cdots x_{i_t}\cdot y_n\cdot{\bf f}_{J_1}(C_1;\sigma))$ has the required expression.
	\smallskip\\
	\textsc{Case 3.3.} Let $y_n\in C$ and $x_n\in\sigma$. Setting $C_3=C\setminus\{x_{i_1},\dots,x_{i_t},x_{j_1},\dots,x_{j_p},y_n\}$, $\sigma_3=\sigma\setminus x_n$ and ${\bf b}=x_{i_1}\cdots x_{i_t}\cdot x_{j_1} \cdots x_{j_p}\cdot x_ny_n\cdot{\bf f}_{J_{1,2}}(C_3;\sigma_3)$, by the mapping cone (Subsection \ref{sub:1}, (ii)), we have that
	$$
	d_i^J({\bf f}(C;\sigma))\ =\ -d_{i-1}^{J_{1,2}}({\bf b})+\psi_{i-1}({\bf b}).
	$$
	Let us compute $-d_{i-1}^{J_{1,2}}({\bf b})$. Since ${\bf z}_{N[x_n]\cup N[y_n]}=x_{i_1}\cdots x_{i_t}\cdot x_{j_1} \cdots x_{j_p}\cdot x_ny_n$, we have
	\small\begin{align*}
	-d_{i-1}^{J_{1,2}}({\bf b})\!&=\!-{\bf z}_{N[x_n]\cup N[y_n]}\!\!\sum_{x_s\in\sigma_3}\!(-1)^{\alpha(\sigma_3;x_s)}\big[y_s{\bf f}_{J_{1,2}}((C_3\setminus y_s)\cup x_s;\sigma_3\setminus x_s)-x_s{\bf f}_{J_{1,2}}(C_3;\sigma_3\setminus x_s)\big]\\
	&=-\sum_{x_s\in\sigma_3}(-1)^{\alpha(\sigma\setminus x_n;x_s)}\big[y_s{\bf f}((C\setminus y_s)\cup x_s;\sigma\setminus x_s)-x_s{\bf f}(C;\sigma\setminus x_s)\big]\\
	&=\phantom{-}\!\!\sum_{x_s\in\sigma\setminus x_n}\!\!\!(-1)^{\alpha(\sigma;x_s)}\big[y_s{\bf f}((C\setminus y_s)\cup x_s;\sigma\setminus x_s)-x_s{\bf f}(C;\sigma\setminus x_s)\big],
	\end{align*}\normalsize
	where the last equation follows from the fact that $\alpha(\sigma;x_s)=\alpha(\sigma\setminus x_n;x_s)+|\{x_n\}|$ for all $x_s\in\sigma\setminus x_n$.
	
	Whereas, for the term ${\psi}_{i-1}({\bf b})$, by the definition of the comparison maps, we have 
	\begin{align*}
	{\psi}_{i-1}({\bf b})\ &=\ {\bf z}_{N[x_n]}\cdot{\bf f}_{J_1}(C_3\cup\{x_{j_1},\dots,x_{j_p}\};\sigma_3) -{\bf z}_{N[y_n]}\cdot{\bf f}_{J_2}(C_3\cup\{x_{i_1},\dots,x_{i_t}\};\sigma_3)\\
	&=\ y_n{\bf f}((C\setminus y_n)\cup x_n;\sigma\setminus x_n)-x_n{\bf f}(C;\sigma\setminus x_n)\\
	&=\ (-1)^{\alpha(\sigma;x_n)}\big[y_n{\bf f}((C\setminus y_n)\cup x_n;\sigma\setminus x_n)-x_n{\bf f}(C;\sigma\setminus x_n)\big],
	\end{align*}
	where the last equation follows because $\alpha(\sigma;x_n)=|\emptyset|=0$. Hence, we see that
	$$
	d_i^J({\bf f}(C;\sigma))\ =\ \sum_{x_s\in\sigma}\!(-1)^{\alpha(\sigma;x_s)}\big[y_s{\bf f}((C\setminus y_s)\cup x_s;\sigma\setminus x_s)-x_s{\bf f}(C;\sigma\setminus x_s)\big],
	$$
	as desired. The induction is complete and the result follows.
\end{proof}

\begin{Corollary}\label{cor:res}
	Let $G$ be a Cohen--Macaulay very well--covered graph. Then 
	\begin{align*}
	\beta_i(I(G)^\vee)\ &=\ \sum_{C\in\mathcal{C}(G)}\binom{|\mathcal{C}(G;C)|}{i},\ \ \ \ i\ge0,\\
	\pd(I(G)^\vee)\ &=\ \max\big\{\big|\mathcal{C}(G;C)\big|\ :\ C\in\mathcal{C}(G)\big\}. 
	\end{align*}
	In particular, the graded Betti numbers of $I(G)^\vee$ do not depend upon the characteristic of the underlying field $K$.
\end{Corollary}

Let $G$ be a graph and let $z_{i}z_{j},z_{k}z_{\ell}\in E(G)$ be a pair of edges. We say that $z_{i}z_{j}$ and $z_{k}z_{\ell}$ are \textit{3--disjoint} if the induced subgraph of $G$ on the vertex set $\{z_i,z_j,z_k,z_\ell\}$ consists of two disjoint edges. The maximum size of a set of pairwise $3$--disjoint edges in $G$ is denoted by $a(G)$ \cite{MMCRTY2011}.

The next result holds.

\begin{Corollary}\label{cor:pdGa(G)}
	Let $G$ be a Cohen--Macaulay very well--covered graph. Then
	$$
	a(G) = \reg(S/I(G))= \max\big\{\big|\mathcal{C}(G;C)\big|\ :\ C\in\mathcal{C}(G)\big\}=\pd(I(G)^\vee).
	$$
\end{Corollary}
\begin{proof} From \cite[Proposition 8.1.10]{JT}, $\pd(I(G)^\vee)=\reg(S/I(G))$. Thus the assertion follows from \cite[Lemma 3.4]{MMCRTY2011} and Corollary \ref{cor:res}. 
\end{proof}
\begin{Remark}\label{rem:pdGa(G)} \em The equality $a(G) = \reg(S/I(G))$ has been firstly proved in \cite[Lemma 3.4]{MMCRTY2011}. Here we give a new proof by Betti splittings. 
By \cite[Lemma 2.2]{Kat2006}, one always have $\reg(S/I(G))\ge a(G)$. It remains to prove that $\pd(I(G)^\vee)\le a(G)$. Assume Setup \ref{Setup:ResVeryWellCG}. 

By \cite[Corollary 2.2]{FHT2009} applied to the Betti splitting provided in Proposition \ref{Prop:BettiSplitVWCG}, and by Lemma \ref{Lem:J12intersect}, we have
	\begin{align*}
	\pd(I(G)^\vee)\ &=\ \max\big\{\pd(I(G_1)^\vee),\ \pd(I(G_2)^\vee),\ \pd(I(G_3)^\vee)+1\big\}.
	\end{align*}
	By induction on $n=|V(G)|/2$, we may assume that $\pd(I(G_i)^\vee)\le a(G_i)$ for $i=1,2,3$. Note that $a(G_i)\le a(G)$, for $i=1,2,3$, clearly. Thus it remains to prove the inequality $a(G_3)+1\le a(G)$. Let $D$ be a set of pairwise $3$--disjoint edges of $G_3$ with $|D|=a(G_3)$. Then $D\cup x_ny_n$ is again a set of pairwise $3$--disjoint edges of $G$. Indeed, if $e=z_kz_\ell\in D$, then $\{z_k,z_{\ell}\}\subseteq V(G_3)=V(G)\setminus\{x_i,y_i:x_i\in N[x_n]\cup N[y_n]\}$. Thus there can not be any edge connecting $z_k$ or $z_\ell$ with either $x_n$ or $y_n$. This shows that the induced subgraph with vertex set $\{z_k,z_\ell,x_n,y_n\}$ has only two edges. Hence $a(G)\ge|D|+1=a(G_3)+1$.
\end{Remark}

Let $M$ be a finitely generated $S$-module of dimension $d$, and let $P_M$ be the Hilbert polynomial of $M$ \cite{BH, JT}. Then, $P_M(t)=\sum_{i=0}^{d-1}(-1)^{d-1-i}e_{d-1-i}\binom{t+i}{i}$, $e_{d-1-i}\in \QQ$, for all $i$. We define the \textit{multiplicity} of $M$ as
$$
e(M)\ =\ \begin{cases}
\hfil e_0&\text{if}\ d>0,\\
\textup{length}(M)&\text{if}\ d=0.
\end{cases}
$$

Now, we verify that $I(G)^\vee$ has the \textit{alternating sum property} and thanks to Corollary \ref{cor:res} we get a formula for the \emph{multiplicity} of $S/I(G)^\vee$, $e(S/I(G)^\vee)$.

We quote next definition from \cite[Definition 4.1]{Zheng2004}.

\begin{Definition}\rm Let $I$ be a monomial ideal in $S$ with $G(I)=\{u_1,\ldots,u_m\}$ and let $d=\min\{\deg(u_i):i=1,\ldots,m\}$. We say that $I$ has the \textit{alternating sum property}, if
$$
\sum_{i\ge 1}(-1)^i\beta_{i, i+j}(S/I)=\begin{cases}
		-1,& \mbox{for $j=d$,}\\
		0, &\mbox{for $j>d$.}
				\end{cases}
$$
\end{Definition}
\begin{Proposition}\label{prop:summult}
	Let $G$ be a Cohen--Macaulay very well--covered graph. Then
	\begin{enumerate}[label=\textup{(\alph*)}]
		\item $\sum_{i\ge1}(-1)^{i}\beta_i(S/I(G)^\vee)= -1$.
		\item $e(S/I(G)^\vee)=|E(G)|=\frac{1}{2}\sum_{C\in\mathcal{C}(G)}\sum_{i=0}^{|\mathcal{C}(G;C)|}(-1)^{i+1}\binom{|\mathcal{C}(G;C)|}{i}(n+i)^2$.
	\end{enumerate}
\end{Proposition}
\begin{proof} Let $|V(G)|=2n$.

		(a). We proceed by induction on $n\ge1$. For $n=1,2$ the only Cohen--Macaulay very well--covered graphs are the following ones.\\
		
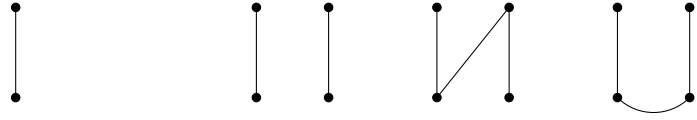
\vspace*{-0.1cm}
		\begin{center}
			\begin{tikzpicture}[scale=0.8]
			\filldraw (-1,0) circle (2pt);
			\filldraw (-1,1.5) circle (2pt);
			\filldraw (3,0) circle (2pt);
			\filldraw (3,1.5) circle (2pt);
			\filldraw (4.2,0) circle (2pt);
			\filldraw (4.2,1.5) circle (2pt);
			\filldraw (6,0) circle (2pt);
			\filldraw (6,1.5) circle (2pt);
			\filldraw (7.2,0) circle (2pt);
			\filldraw (7.2,1.5) circle (2pt);
			\filldraw (9,0) circle (2pt);
			\filldraw (9,1.5) circle (2pt);
			\filldraw (10.2,0) circle (2pt);
			\filldraw (10.2,1.5) circle (2pt);
			\draw[-] (-1,0) -- (-1,1.5);
			\draw[-] (3,0) -- (3,1.5);
			\draw[-] (4.2,0) -- (4.2,1.5);
			\draw[-] (6,0) -- (6,1.5);
			\draw[-] (6,0) -- (7.2,1.5);
			\draw[-] (7.2,0) -- (7.2,1.5);
			\draw[-] (9,0) -- (9,1.5);
			\draw[-] (10.2,0) -- (10.2,1.5);
			\draw (9,0) to [bend right=45] (10.2,0);
			\end{tikzpicture}
		\end{center}\vspace*{-0.2cm}
		For each of these graphs, (a) holds. Let $n>2$. Using the same notation as in Setup \ref{Setup:ResVeryWellCG}, from Proposition \ref{Prop:BettiSplitVWCG}, Lemma \ref{Lem:J12intersect} 
		and by the inductive hypothesis on the graphs $G_1,G_2,G_3$, we have that
		\begin{align*}
		& \sum_{i\ge1}(-1)^{i}\beta_i(S/I(G)^\vee) = \sum_{i\ge0}(-1)^{i}\big[\beta_i(S/I(G_1)^\vee)+\beta_i(S/I(G_2)^\vee)+\beta_{i-1}(S/I(G_3)^\vee)\big]\\
		=&\ \sum_{i\ge0}(-1)^{i}\beta_i(S/I(G_1)^\vee)+\sum_{i\ge0}(-1)^{i}\beta_i(S/I(G_2)^\vee)-\sum_{i\ge0}(-1)^{i}\beta_{i}(S/I(G_3)^\vee)\ =\\
		=&\ \ -1-1+1\ =\ -1.
		\end{align*}
		(b). Since $I(G)^\vee=\bigcap_{z_iz_j\in E(G)}(z_i,z_j)$ we see that $I(G)^\vee$ has height two and, furthermore, $e(S/I(G)^\vee)=|E(G)|$ by \cite[Corollary 6.2.3]{JT}. The first equality holds true.
		
		On the other hand, by a formula of Peskine and Szpiro \cite{PS1974} (see also \cite[Corollary 6.1.7]{JT}), since $\beta_i(S/I(G)^\vee)=\beta_{i-1}(I(G)^\vee)$, one has that
		\begin{eqnarray*}
e(S/I(G)^\vee)&=&\frac{(-1)^h}{h!}\sum_{i=1}^{\pd(S/I(G)^\vee)}(-1)^{i+1}\sum_{j=1}^{\beta_i(S/I(G)^\vee)}a_{ij}^h=\\
&=&\frac{(-1)^h}{h!}\sum_{i=0}^{\pd(I(G)^\vee)}(-1)^{i+1}\sum_{j=1}^{\beta_i(I(G)^\vee)}a_{ij}^h,
\end{eqnarray*}
where $h=\alt(I(G)^\vee)$ and $a_{ij}$ are the shifts of the $i$th free module of the minimal free resolution $\FF$ of $S/I(G)^\vee$, \emph{i.e.}, 
			$\bigoplus_{j=1}^{\beta_i(S/I(G)^\vee)}S(-a_{ij})$. Since $I(G)^\vee$ has a $n$--linear resolution and $\alt(I(G)^\vee)=h=2$, from Corollary \ref{cor:res}, we have that
		\begin{align*}
		e(S/I(G)^\vee)\ &=\ \frac{1}{2}\sum_{i\ge0}(-1)^{i+1}\sum_{C\in\mathcal{C}(G)}\binom{|\mathcal{C}(G;C)|}{i}(n+i)^2\ =\\
		&=\ \frac{1}{2}\sum_{C\in\mathcal{C}(G)}\sum_{i=0}^{|\mathcal{C}(G;C)|}(-1)^{i+1}\binom{|\mathcal{C}(G;C)|}{i}(n+i)^2.\quad\quad\square
		\end{align*}

\end{proof}

\begin{Example} \em Let us consider the Cohen--Macaulay well--covered graph $G$ with $12$ vertices in Example \ref{expl:main}.
Set $C_1=\{x_1,x_2,x_3,x_4,x_5,x_6\}$, $C_2=\{y_1,x_2,x_3,x_4,x_5,x_6\}$, $C_3=\{x_1,y_2,x_3,x_4,x_5,x_6\}$, $C_4=\{x_1,x_2,x_3,x_4,y_5,x_6\}$, $C_5=\{x_1,x_2,x_3,x_4,x_5,y_6\}$, $C_6=\{x_1,x_2,x_3,x_4,y_5,y_6\}$, $C_7=\{x_1,x_2,x_3,y_4,y_5,y_6\}$ and $C_8=\{x_1,x_2,y_3,y_4,y_5,y_6\}$. Then
\begin{align*}
&\begin{array}{ccc}
\bottomrule[1.05pt]
\rowcolor{black!20}
\phantom{-}\mathcal{C}(G)\phantom{-}&\mathcal{C}(G;C)&\vert\mathcal{C}(G;C)\vert\\
\toprule[1.05pt]C_1&\emptyset&0\\
C_2&\{x_1\}&1\\
C_3&\{x_2\}&1\\
C_4&\{x_5\}&1\\
C_5&\{x_6\}&1\\
C_6&\{x_5, x_6\}&2\\
C_7&\{x_4\}&1\\
C_8&\{x_3\}&1\\
\toprule[1.05pt]
\end{array}
\end{align*}

Using \cite{GDS}, one can verify that $\beta_1(S/I(G)^\vee) =8$, $\beta_2(S/I(G)^\vee) =8$, $\beta_3(S/I(G)^\vee) =1$. Hence, $\sum_{i\ge1}(-1)^{i}\beta_i(S/I(G)^\vee)= -1$ and the alternating sum property holds. Moreover, $a(G) = \reg(S/I(G))= 2= \max\big\{\big|\mathcal{C}(G;C)\big|\ :\ C\in\mathcal{C}(G)\big\}=\pd(I(G)^\vee)$.

Finally, let us compute the multiplicity of $S/I(G)^\vee$. Firstly, we have $e(S/I(G)^\vee)=|E(G)|= 20$. On the other hand, from the previous table, one has
\begin{align*}
&\frac{1}{2}\sum_{C\in\mathcal{C}(G)}\sum_{i=0}^{|\mathcal{C}(G;C)|}(-1)^{i+1}\binom{|\mathcal{C}(G;C)|}{i}(6+i)^2=\\
&= -\frac{1}{2}\binom{|\mathcal{C}(G;C_1)|}{0}6^2+ \frac{1}{2}\sum_{C\in\mathcal{C}(G)\setminus (C_1\cup C_6)}\sum_{i=0}^{|\mathcal{C}(G;C)|}(-1)^{i+1}\binom{|\mathcal{C}(G;C)|}{i}(6+i)^2+\\
&+\frac{1}{2}\sum_{i=0}^{|\mathcal{C}(G;C_6)|}(-1)^{i+1}\binom{|\mathcal{C}(G;C_6)|}{i}(6+i)^2 =\\
&=\frac{1}{2}\Big[-\binom{0}{0}6^2+ 6\sum_{i=0}^{1}(-1)^{i+1}\binom{1}{i}(6+i)^2+\sum_{i=0}^{2}(-1)^{i+1}\binom{2}{i}(6+i)^2\Big]=\\
&=\frac{1}{2}(-36+78-2) =20
\end{align*}
and Proposition \ref{prop:summult} holds true.

\end{Example}

\section{Homological shifts of vertex cover ideals of very well--covered graphs and applications to Hibi ideals}\label{sec:4}

In this section, we investigate the homological shift ideals of powers of the vertex cover ideal of Cohen--Macaulay very well--covered graphs. Here is our main result.

In what follows we refer to the notation in Setup \ref{Setup:ResVeryWellCG}.

\begin{Theorem}\label{Thm:HSCMVWCG}
	Let $G$ be a Cohen--Macaulay very well--covered graph with $2n$ vertices. Then $\HS_k(I(G)^\vee)$ has linear quotients with respect to the lexicographic order $>_{\lex}$ induced by $x_n>y_n>x_{n-1}>y_{n-1}>\dots>x_1>y_1$, for all $k\ge0$.
\end{Theorem}
\begin{proof}
	We proceed by strong induction on $n=|V(G)|/2\ge1$. For $n=1$, $I(G)=(x_1y_1)$ and $J=I(G)^\vee=(x_1,y_1)$. Hence, $\HS_0(J)=J$, $\HS_1(J)=(x_1y_1)$ have linear quotients with respect to $>_{\lex}$. 
	
	Let $n>1$. By Proposition \ref{Prop:BettiSplitVWCG}, 
	\begin{equation}\label{eq:BettiSplittingHSVeryWellCG}
	I(G)^\vee\ =\  x_{j_1}x_{j_2}\cdots x_{j_p}\cdot x_n I(G_2)^\vee + x_{i_1}x_{i_2}\cdots x_{i_t}\cdot y_n I(G_1)^\vee
	\end{equation}
	is a Betti splitting.
	
	Set $f=x_{j_1}\cdots x_{j_p}\cdot x_n$ and $g=x_{i_1}\cdots x_{i_t}\cdot y_n$. Since $\HS_k(wI)=w\cdot\HS_k(I)$ for all monomial ideals $I$ in $S$ and all non zero monomials $w\in S$, from (\ref{eq:BettiSplittingHSVeryWellCG}), Proposition \ref{Prop:HSBettiSplitting} and Lemma \ref{Lem:J12intersect}, 
	we have that
	\begin{align*}
	\HS_k(I(G)^\vee)\ =&\ fg\cdot\HS_{k-1}(I(G_3)^\vee)+f\cdot\HS_k(I(G_2)^\vee)+g\cdot\HS_{k}(I(G_1)^\vee).
	\end{align*}
	
	Note that, since $|V(G_j)|<|V(G)|$ ($j=1,2,3$), by the inductive hypothesis, $\HS_k(I(G_j)^\vee)$ have linear quotients with respect to $>_{\lex}$.
	
	Let 
	\begin{align}
	\label{eq:orderHSG_3}G(\HS_{k-1}(I(G_3)^\vee))\ &=\ \big\{w_1>_{\lex}w_2>_{\lex}\dots>_{\lex}w_q\big\},\\
	\label{eq:orderHSG_2}G(\HS_k(I(G_2)^\vee))\ &=\ \big\{v_1\,>_{\lex}\,v_2>_{\lex}\dots>_{\lex}v_r\big\},\\
	\label{eq:orderHSG_1}G(\HS_k(I(G_1)^\vee))\ &=\ \big\{u_1>_{\lex}u_2\!>_{\lex}\dots>_{\lex}u_s\big\},
	\end{align}
	thus  $G(\HS_k(I(G)^\vee))$ is ordered as follows:
	$$
	fgw_1>_{\lex}\dots>_{\lex}fgw_q\ \ >_{\lex}\ \ fv_1>_{\lex}\dots>_{\lex}fv_r\ \ >_{\lex}\ \  gu_1>_{\lex}\dots>_{\lex}gu_s.
	$$
	
	We prove that such an order is an admissible order of $\HS_k(I(G)^\vee)$. For our purpose, we need to show that all the following three colon ideals
	\begin{align}
\label{eq:colonfgw}&(fgw_1,\dots,fgw_{\ell-1}):fgw_{\ell}, \,\,\,\, \ell\in\{2,\dots,q\},\\
	\label{eq:colonfgfv}&\big(fg\HS_{k-1}(I(G_3)^\vee),fv_1,\dots,fv_{\ell-1}\big):fv_{\ell}, \,\,\,\, \ell\in\{1,\dots,r\},\\
	\label{eq:colonfgfgu}&\big(fg\HS_{k-1}(I(G_3)^\vee),f\HS_k(I(G_2)^\vee),gu_1,\dots,gu_{\ell-1}\big):gu_{\ell},\,\,\,\, \ell\in\{1,\dots,s\},
	\end{align}
	are generated by variables.
	\bigskip\\
	\textsc{First Colon Ideal.} 
	Let us consider the colon ideal in (\ref{eq:colonfgw}). Since, $$(fgw_1,\dots,fgw_{\ell-1}):fgw_{\ell}=(w_1,\dots,w_{\ell-1}):w_{\ell},$$ the assertion follows from the fact that $\HS_{k-1}(I(G_3)^\vee)$ has linear quotients with respect to the order in (\ref{eq:orderHSG_3}). 
		\medskip\\
	\textsc{Second Colon Ideal.} 
	Let us consider the colon ideal in (\ref{eq:colonfgfv}). \\Set $P= \big(fg\HS_{k-1}(I(G_3)^\vee),fv_1,\dots,fv_{\ell-1}\big):fv_{\ell}$. One can observe that 
		\[P= (fg\HS_{k-1}(I(G_3)^\vee)):fv_{\ell}+(v_1,\dots,v_{\ell-1}):v_{\ell},\]
	with $(v_1,\dots,v_{\ell-1}):v_{\ell}$ generated by variables. Indeed, $\HS_{k}(I(G_2)^\vee)$ has linear quotients with respect to the order in (\ref{eq:orderHSG_2}).\\
	 Thus, if we show that each generator of $(fg\HS_{k-1}(I(G_3)^\vee)):fv_{\ell}$ is divided by a variable of $P$, we conclude that $P$ is generated by variables, as wanted. The colon ideal $(fg\HS_{k-1}(I(G_3)^\vee)):fv_{\ell}$ is generated by the monomials
	$$
	\frac{\lcm(fgw_j,fv_{\ell})}{fv_{\ell}},\ \ \ j=1,\dots,q.
	$$
	Fix $j\in\{1,\dots,q\}$. By Construction \ref{Constr:MinimalFreeResVeryWellCG}, we have that
	\begin{align*}
	fgw_j\ &=\ w{\bf x}_{\sigma},&& w\in G(I(G)^\vee),\ y_n\in\supp(w),\ \sigma\subseteq\mathcal{C}(G;\supp(w)),\ x_n\in\sigma,\ |\sigma|=k,\\
	fv_{\ell}\ &=\ v{\bf x}_{\tau},&& v\in G(I(G)^\vee),\ x_n\in\supp(v),\ \tau\subseteq\mathcal{C}(G;\supp(v)),\ x_n\notin\tau,\ |\tau|=k.
	\end{align*}
	Let $h=\lcm(w{\bf x}_{\sigma},v{\bf x}_{\tau})/(v{\bf x}_{\tau})$. If $\deg(h)=1$, $h$ is a variable and there is nothing to prove. Suppose $\deg(h)>1$. Let us consider the following integer
	\begin{equation}\label{eq:integerI}
	i\ =\ \min\Big\{i\ :\ z_i\ \text{divides }h=\frac{\lcm(w{\bf x}_{\sigma},v{\bf x}_{\tau})}{v{\bf x}_{\tau}},\ z_i\in\{x_i,y_i\}\Big\}.
	\end{equation}
	Since $\deg(h)>1$ and $y_n$ divides $h$, we have that $i<n$. From Lemma \ref{Lem:minimalGeneratorsI(G)^vee}, we have that $y_i$ divides at least one of the monomials $w{\bf x}_{\sigma}$, $v{\bf x}_{\tau}$. Indeed, if $y_i$ does not divide any of these monomials, then $x_i$ will divide both these monomials,  and consequentially, $z_i\in\{x_i,y_i\}$ does not divide $h$. Against our assumption. \\
Now, let $N[y_i]=\{x_{k_1},\dots,x_{k_b},x_i,y_i\}$. We claim that $x_{k_1}\cdots x_{k_b}$ divides both monomials $w$, $v$. Indeed, $x_i$ divides at least one of the monomials $w$, $v$. For instance, say $x_i$ divides $w{\bf x}_{\sigma}$, then either $x_i$ divides $w$ or $x_i(w/y_i)\in G(I(G)^\vee)$. In both cases, $x_{k_1}\cdots x_{k_b}$ must divide $w$, since $w$ or $x_i(w/y_i)$ are minimal vertex covers. The same reasoning works if $x_i$ divides $v{\bf x}_{\tau}$. Thus $x_{k_1}\cdots x_{k_b}$ must divide $w$ or $v$. Without loss of generality, suppose that $x_{k_1}\cdots x_{k_b}$ divides $w$. Then, by Construction \ref{Constr:MinimalFreeResVeryWellCG}, none of the variables $y_{k_1},\dots,y_{k_b}$ divides $w$. By Characterization \ref{char:veryWellCGCM}(iii), $k_1<\dots<k_b<i$. Thus, by the meaning of $i$, $z_d\in\{x_d,y_d\}$ does not divide $h$ for all $d=k_1,\dots,k_b$. Hence, we see that $x_{k_1}\cdots x_{k_b}$ divides both the monomials $w$, $v$. Now, we distinguish two cases.
	\smallskip\\
	\textsc{Case 1.} Suppose $x_iy_i$ divides one between the monomials $w{\bf x}_{\sigma}$ and $v{\bf x}_{\tau}$, and that $y_i$ divides the other one. Suppose, for instance, that $x_iy_i$ divides $w{\bf x}_{\sigma}$ and that $y_i$ divides $v{\bf x}_{\tau}$. Then, $x_i(v/y_i){\bf x}_{\tau}>_{\lex}v{\bf x}_{\tau}$. Moreover, $x_i(v/y_i)\in G(I(G)^\vee)$, because $x_{k_1}\cdots x_{k_b}$ divides $v$ and so $\supp(x_i(v/y_i))$ is again a minimal vertex cover of $G$.
	\\ We claim that $x_i(v/y_i){\bf x}_{\tau}\in\HS_{k}(I(G)^\vee)$. Indeed for all $c\in\tau$, $N[y_c]\setminus\{x_c,y_c\}\subseteq\supp(v)$. Since $N[y_c]\setminus\{x_c,y_c\}$ is a subset of $X=\{x_1,\dots,x_n\}$ and $\supp(x_i(v/y_i))=(\supp(v)\setminus y_i)\cup x_i$, we see that $N[y_c]\setminus\{x_c,y_c\}$ is again a subset of $\supp(x_i(v/y_i))$. Thus, for all $c\in\tau$, $x_cx_i(v/y_i)/y_c\in G(I(G)^\vee)$, \emph{i.e.}, $\tau\subseteq\mathcal{C}(G;\supp(x_i(v/y_i)))$. Finally, $x_i(v/y_i){\bf x}_{\tau}\in\HS_{k}(I(G)^\vee)$. Moreover,
	$$
	\frac{\lcm(x_i(v/y_i){\bf x}_{\tau},v{\bf x}_{\tau})}{v{\bf x}_{\tau}}=x_i\ \ \text{and}\ \ x_i\ \ \text{divides}\ \ h.
	$$
	It follows that $h$ is divided by the variable $x_i$ belonging to $P$, as desired.
	\smallskip\\
	Otherwise, suppose $x_iy_i$ divides $v{\bf x}_{\tau}$ and $y_i$ divides $w{\bf x}_{\sigma}$. Then, $x_i(w/y_i){\bf x}_{\sigma}>_{\lex}w{\bf x}_{\sigma}$ and, as before, $x_i(w/y_i){\bf x}_{\sigma}\in G(\HS_k(I(G)^\vee))$. Moreover,
	$$
	h'=\frac{\lcm(x_i(w/y_i){\bf x}_{\sigma},v{\bf x}_{\tau})}{v{\bf x}_{\tau}}=h/x_i\ \ \text{divides}\ \ h
	$$
	and so $h'\in P$ and $\deg(h')<\deg(h)$. Hence, we can iterate the previous reasoning by considering the integer $i'$ arising from $h'$, as in formula (\ref{eq:integerI}). In such a situation it is $i'>i$.
	\smallskip\\
	\textsc{Case 2.} Suppose $x_i$ divides one of the monomials $w{\bf x}_{\sigma}$, $v{\bf x}_{\tau}$ and $y_i$ divides the other one. Suppose $x_i$ divides $w{\bf x}_{\sigma}$ and $y_i$ divides $v{\bf x}_{\tau}$. Then $x_i(v/y_i){\bf x}_{\tau}>_{\lex}v{\bf x}_{\tau}$, and arguing as before, one gets $x_i(v/y_i){\bf x}_{\tau}\in\HS_{k}(I(G)^\vee)$. Moreover,
	$$
	\frac{\lcm(x_i(v/y_i){\bf x}_{\tau},v{\bf x}_{\tau})}{v{\bf x}_{\tau}}=x_i\ \ \text{and}\ \ x_i\ \ \text{divides}\ \ h.
	$$
	Therefore, $h$ is divided by the variable $x_i$ belonging to $P$, as desired.
	\smallskip\\
	Otherwise, suppose $x_i$ divides $v{\bf x}_{\tau}$ and $y_i$ divides $w{\bf x}_{\sigma}$. It follows that $x_i(w/y_i){\bf x}_{\sigma}\in G(\HS_k(I(G)^\vee))$ and $x_i(w/y_i){\bf x}_{\sigma}>_{\lex}w{\bf x}_{\sigma}$. Moreover,
	$$
	h'=\frac{\lcm(x_i(w/y_i){\bf x}_{\sigma},v{\bf x}_{\tau})}{v{\bf x}_{\tau}}=h/x_i\ \ \text{divides}\ \ h.
	$$
	Thus $h'\in P$ and $\deg(h')<\deg(h)$.  In such a case, we iterate the reasoning made above by considering the integer $i'$ arising from $h'$ as in (\ref{eq:integerI}), and this time $i'>i$. It is clear that iterating the reasoning above we get that $P$ is generated by variables.
	\medskip\\
	\textsc{Third Colon Ideal.} Set $P\!=\! \big(fg\HS_{k-1}(I(G_3)^\vee),f\HS_k(I(G_2)^\vee),gu_1,\dots,gu_{\ell-1}\big)\!:\!gu_{\ell}$. One can observe that $P$ is generated by the monomials
	$$
	\frac{\lcm(fgw_i,gu_{\ell})}{gu_{\ell}},\ \ \frac{\lcm(fv_j,gu_{\ell})}{gu_{\ell}},\ \ \frac{\lcm(gu_h,gu_{\ell})}{gu_{\ell}},\ \ 
	$$
	with $i\in[q]$, $j\in[r]$ and $h\in[\ell-1]$. To prove that $P$ is generated by variables it is enough to show that any of the monomials above is divided by a variable $z_b\in P$. Note that $\HS_{k}(I(G_1)^\vee)$ has linear quotients with respect to the order in (\ref{eq:orderHSG_1}), thus the colon ideal $(gu_1,\dots,gu_{\ell-1}):gu_{\ell}=(u_1,\dots,u_{\ell-1}):u_{\ell}$ is generated by variables. Thus, $\lcm(gu_h,gu_{\ell})/(gu_{\ell})\in (u_1,\dots,u_{\ell-1}):u_{\ell}$ is divided by a variable of $P$. For the other two type of monomial generators it suffices to repeat the same argument as in the \textsc{Second Colon Ideal} case.
\end{proof}

As an immediate consequence we have the following classification.
\begin{Theorem}\label{Thm:ClassificationHomLin}
	Let $G$ be a very well--covered graph with $2n$ vertices. Then, the following conditions are equivalent:
	\begin{enumerate}[label=\textup{(\roman*)}]
		\item $G$ is Cohen--Macaulay.
		\item $I(G)^\vee$ has homological linear quotients.
		\item $I(G)^\vee$ has homological linear resolution.
	\end{enumerate}
\end{Theorem}

Let $G$ be a finite simple graph with vertex set $V(G)=X=\{x_1,\dots,x_n\}$. The \textit{whisker graph} $G^*$ of $G$ is the graph with vertex set $V(G^*)=\{x_1,\dots,x_n\}\cup\{y_1,\dots,y_n\}$ and edge set $E(G^*)=E(G)\cup\{x_iy_i:i=1,\dots,n\}$. In other words, the whisker graph of $G$ is obtained by adding to each vertex $x_i$ a \textit{whisker}, that means that for each vertex $x_i\in V(G)$ we add a new vertex $y_i$ and an edge $x_iy_i$ connecting these two vertices. 

\begin{Corollary}\label{cor:HSwhisker}
	Let $G$ be any simple graph with $n$ vertices. Then the vertex cover ideal of the whisker graph $G^*$ of $G$ has homological linear quotients.
\end{Corollary}
\begin{proof}
	It is easy to see that $G^*$ satisfies the conditions (i)--(v) of Characterization \ref{char:veryWellCGCM}. Thus $G^*$ is a Cohen--Macaulay very well--covered graph and the result follows from the previous theorem.
\end{proof}

Our experiments and the results above, suggest the following conjecture.
\begin{Conjecture}\label{Conj:PowersHSCMverywell}
	\rm Let $G$ be a Cohen--Macaulay very well--covered graph with 2n vertices. Then $\HS_k((I(G)^\vee)^\ell)$ has linear quotients with respect to the lexicographic order induced by $x_n>y_n>x_{n-1}>y_{n-1}>\dots>x_1>y_1$, for all $k\ge0$, and all $\ell\ge1$.
\end{Conjecture}

Note that our conjecture would also imply that $I(G)^\vee$ has linear powers.\\

At present it seems too difficult to prove our conjecture in full generality. Therefore we concentrate our attention on the subclass of Cohen--Macaulay \textit{bipartite} graphs. 

For this purpose, we need to recall what an \textit{Hibi ideal} is \cite{Hibi87}.

Let $(P,\succeq)$ be a finite partially ordered set (a \textit{poset}, for short) and set $P=\{p_1, \ldots, p_n\}$. A \textit{poset ideal} of $P$ is a subset $\mathcal{I}$ of $P$ such that if 
$p_i\in P$, $p_j\in\mathcal{I}$ and $p_i\preceq p_j$, then $p_i\in\mathcal{I}$ \cite{HH2005a}. To any poset ideal $\mathcal{I}$ of $P$, we associate the squarefree monomial 
$u_\mathcal{I}=(\prod_{p_i\in\mathcal{I}}x_i)(\prod_{p_i\in P\setminus\mathcal{I}}y_i)$ in the polynomial ring $S=K[x_1, \ldots, x_n, y_1, \ldots, y_n]$. Then the \textit{Hibi ideal} (associated to $P$) is the monomial ideal of $S$ defined as follows: 
$$
H_P\ =\ \big(u_\mathcal{I}\ :\ \mathcal{I}\ \text{is a poset ideal of}\ P\big).
$$

As an immediate consequence of Theorem \ref{Thm:HSCMVWCG} we have the next result.
\begin{Corollary}
	Let $(P,\succeq)$ be a finite poset. Then $H_P$ has homological linear quotients.
\end{Corollary}
\begin{proof}
	By \cite[Lemma 9.1.11]{JT}, the Alexander dual $H_P^\vee$ may be seen as the edge ideal of a Cohen--Macaulay bipartite graph $G_P$. In particular $G_P$ is a Cohen--Macaulay very well--covered graph. Hence, seeing $H_P$ as the cover ideal of $G_P$, the result follows immediately from Theorem \ref{Thm:HSCMVWCG}.
\end{proof}

Now, we turn to the powers of an Hibi ideal. We denote by $\mathcal{J}(P)$ the distributive lattice consisting of all poset ideals of $(P,\succeq)$ ordered by inclusion. 

From now on, with abuse of language but with the aim of simplifying the notation, we identify each $p_i\in P$ with the variable $x_i$, $i\in [n]$.  

\begin{Construction}\label{Constr:NewPosetP(ell)}
	\rm Let $(P,\succeq)$ be a finite poset. For any integer $\ell\ge1$, we construct a new poset $(P(\ell),\succeq_\ell)$ defined as follows:
		\begin{enumerate}
		\item[-] $P(\ell)\ =\ \big\{x_{i,1},x_{i,2},\dots,x_{i,\ell}\ :\ i=1,\dots,n\big\}$,
		\item[-] and $x_{i,r}\succeq_\ell x_{j,s}$ if and only if $x_i\succeq x_j$ and $r\ge s$.
	\end{enumerate}
\end{Construction}

By \cite[pag 99]{Hibi87} we have the following useful property.
\begin{Lemma}\label{lem:Hibi}
	Let $(P,\succeq)$ be a finite poset. Each minimal monomial generator of $H_P^\ell$ posses a unique expression of the form
	$$
	u_{\mathcal{I}_1}u_{\mathcal{I}_2}\cdots u_{\mathcal{I}_\ell},\ \ \ \text{with}\ \ \ \mathcal{I}_1\subseteq \mathcal{I}_2\subseteq\dots\subseteq \mathcal{I}_\ell,\ \ \ \mathcal{I}_i\in\mathcal{J}(P),\ \ \ i=1,\dots,\ell. 
	$$
\end{Lemma}

Now, we need the technique of polarization discussed in Subsection \ref{sub:2}. Let $(P,\succeq)$ be a finite poset and $\ell\ge1$ be a positive integer. Note that $X$ and $\emptyset$ are both poset ideals of $P$, thus both $u_{X}=x_1x_2\cdots x_n$ and $u_{\emptyset}=y_1y_2\cdots y_n$ belong to $H_P$. Thus, we have that $u_{X}^\ell=x_1^\ell x_2^\ell\cdots x_n^\ell$ and $u_\emptyset^\ell=y_1^\ell y_2^\ell\cdots y_n^\ell$ belong to $H_P^\ell$. Therefore the polynomial ring in which $H_P^\ell$ lives is
$$
R=K[x_{i,j},y_{i,j}\ :\ i=1,\dots,n,\ j=1,\dots,\ell].
$$

In order to preserve the structure of Hibi ideals, we innocuously \emph{modify} polarization. More precisely,
let $1\le k\le\ell$ and $i\in[n]$, then we set
\begin{align*}
(x_i^k)^\wp\ &=\ x_{i,1}x_{i,2}\cdots x_{i,\ell},\\
(y_i^k)^\wp\ &=\ y_{i,\ell}y_{i,\ell-1}\cdots y_{i,\ell+1-k},
\end{align*}
and extend the polarization of an arbitrary monomial in the obvious way. In other words, with respect to the usual polarization, we are just applying the relabeling of the variables $y_{i,j}\mapsto y_{i,\ell+1-j}$ for $i=1,\dots,n$ and $j=1,\dots,\ell$. 

\begin{Example}
	\rm Consider the poset $(P,\succeq)$ with $P=\{x_1,x_2,x_3\}$, 
	$x_3\succ x_1$ and $x_3\succ x_2$. The poset $(P,\succeq)$ and the distributive lattice $\mathcal{J}(P)$ are depicted below:
	
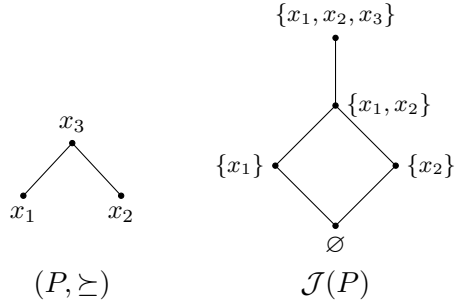
\begin{center}
		\begin{tikzpicture}[scale=0.5]
		\filldraw (0,0) circle (2pt) node[below] {\footnotesize$x_1$};
		\filldraw (1.3,1.4) circle (2pt) node[above] {\footnotesize$x_3$};
		\filldraw (2.6,0) circle (2pt) node[below] {\footnotesize$x_2$};
		\draw[-] (0,0) -- (1.3,1.4) -- (2.6,0);
		\filldraw (6.7,0.8) circle (2pt) node[left] {\scriptsize$\{x_1\}$};
		\filldraw (8.3,2.4) circle (2pt) node[right] {\scriptsize$\{x_1,x_2\}$};
		\filldraw (8.3,4.2) circle (2pt) node[above] {\scriptsize$\{x_1,x_2,x_3\}$};
		\filldraw (9.9,0.8) circle (2pt) node[right] {\scriptsize$\{x_2\}$};
		\filldraw (8.3,-0.8) circle (2pt) node[below] {\footnotesize$\emptyset$};
		\draw[-] (6.7,0.8) -- (8.3,2.4) -- (9.9,0.8) -- (8.3,-0.8) -- (6.7,0.8);
		\draw[-] (8.3,2.4) -- (8.3,4.2);
		\filldraw (1.3,-1.7) node[below] {\small$(P,\succeq)$};
		\filldraw (8.3,-1.7) node[below] {\small$\mathcal{J}(P)$};
		\end{tikzpicture}
	\end{center}
	
The poset $(P(2),\succeq_2)$ and the distributive lattice $\mathcal{J}(P(2))$ are the following ones: \medskip
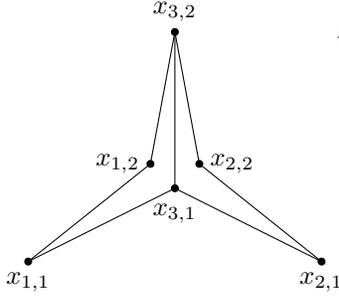
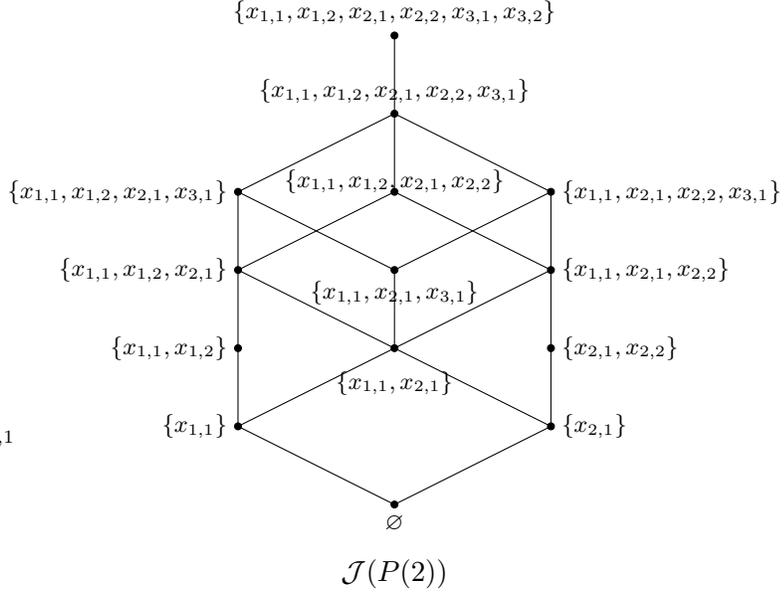
\begin{center}
	\begin{tikzpicture}[scale=0.65]
	\filldraw (1,-1) circle (2pt) node[below] {\footnotesize$x_{1,1}$};
	\filldraw (7,-1) circle (2pt) node[below] {\footnotesize$x_{2,1}$};
	\filldraw (3.5,1) circle (2pt) node[left] {\footnotesize$x_{1,2}$};
	\filldraw (4.5,1) circle (2pt) node[right] {\footnotesize$x_{2,2}$};
	\filldraw (4,0.5) circle (2pt) node[below,yshift=-1.8pt] {\footnotesize$x_{3,1}$};
	\filldraw (4,3.7) circle (2pt) node[above] {\footnotesize$x_{3,2}$};
	\draw[-] (1,-1) -- (3.5,1);
	\draw[-] (4,0.5) -- (7,-1);
	\draw[-] (7,-1) -- (4.5,1);
	\draw[-] (3.5,1) -- (4,3.7);
	\draw[-] (4.5,1) -- (4,3.7);
	\draw[-] (1,-1) -- (4,0.5);
	\draw[-] (4,0.5) -- (4,3.7);
	\filldraw (4,-3.7) node[below] {\small$(P(2),\succeq_2)$};
	\filldraw (15.2,6.8) circle (2pt) node[above] {\scriptsize$\{x_{1,1},x_{1,2},x_{2,1},x_{2,2},x_{3,1},x_{3,2}\}$};
	\filldraw (15.2,5.2) circle (2pt) node[above] {\scriptsize$\{x_{1,1},x_{1,2},x_{2,1},x_{2,2},x_{3,1}\}$};
	\filldraw (12,3.6) circle (2pt) node[left] {\scriptsize$\{x_{1,1},x_{1,2},x_{2,1},x_{3,1}\}$};
	\filldraw (18.4,3.6) circle (2pt) node[right] {\scriptsize$\{x_{1,1},x_{2,1},x_{2,2},x_{3,1}\}$};
	\filldraw (15.2,3.6) circle (2pt) node[above,yshift=-4.5] {\scriptsize$\{x_{1,1},x_{1,2},x_{2,1},x_{2,2}\}$};
	\filldraw (12,2) circle (2pt) node[left] {\scriptsize$\{x_{1,1},x_{1,2},x_{2,1}\}$};
	\filldraw (18.4,2) circle (2pt) node[right] {\scriptsize$\{x_{1,1},x_{2,1},x_{2,2}\}$};
	\filldraw (15.2,2) circle (2pt) node[below] {\scriptsize$\{x_{1,1},x_{2,1},x_{3,1}\}$};
	\filldraw (12,0.4) circle (2pt) node[left] {\scriptsize$\{x_{1,1},x_{1,2}\}$};
	\filldraw (15.2,0.4) circle (2pt) node[below,yshift=-4.5] {\scriptsize$\{x_{1,1},x_{2,1}\}$};
	\filldraw (18.4,0.4) circle (2pt) node[right] {\scriptsize$\{x_{2,1},x_{2,2}\}$};
	\filldraw (12,-1.2) circle (2pt) node[left] {\scriptsize$\{x_{1,1}\}$};
	\filldraw (18.4,-1.2) circle (2pt) node[right] {\scriptsize$\{x_{2,1}\}$};
	\filldraw (15.2,-2.8) circle (2pt) node[below] {\scriptsize$\emptyset$};
	\draw[-] (15.2,-2.8) -- (18.4,-1.2) -- (18.4,3.6) -- (15.2,5.2) -- (12,3.6) -- (12,-1.2) -- (15.2,-2.8);
	\draw[-] (12,2) -- (18.4,-1.2);
	\draw[-] (12,3.6) -- (15.2,2) -- (18.4,3.6);
	\draw[-] (12,-1.2) -- (18.4,2) -- (15.2,3.6) -- (12,2);
	\draw[-] (15.2,6.8) -- (15.2,3.6);
	\draw[-] (15.2,0.4) -- (15.2,2);
	\filldraw (15.2,-3.7) node[below] {\small$\mathcal{J}(P(2))$};
	\end{tikzpicture}
\end{center}\vspace*{-0.4cm}
	We have that
	
	\begin{align*}
	H_P\ &=\ (x_1x_2x_3,\ x_1x_2y_3,\ x_1y_2y_3,\ y_1x_2y_3,\ y_1y_2y_3),\\[5pt]
	H_{P(2)}\ & =\ \big(x_{1,1}x_{1,2}x_{2,1}x_{2,2}x_{3,1}x_{3,2},\ x_{1,1}x_{1,2}x_{2,1}x_{2,2}x_{3,1}y_{3,2},\ x_{1,1}x_{1,2}x_{2,1}x_{2,2}y_{3,1}y_{3,2},\\
	&\phantom{ =\ \big(.}x_{1,1}x_{1,2}x_{2,1}y_{2,2}x_{3,1}y_{3,2},\ x_{1,1}y_{1,2}x_{2,1}y_{2,2}x_{3,1}y_{3,2},\ x_{1,1}y_{1,2}x_{2,1}x_{2,2}x_{3,1}y_{3,2},\\
	&\phantom{ =\ \big(.}x_{1,1}x_{1,2}x_{2,1}y_{2,2}y_{3,1}y_{3,2},\ x_{1,1}y_{1,2}x_{2,1}y_{2,2}y_{3,1}y_{3,2},\ x_{1,1}y_{1,2}x_{2,1}x_{2,2}y_{3,1}y_{3,2},\\
	&\phantom{ =\ \big(.}x_{1,1}x_{1,2}y_{2,1}y_{2,2}y_{3,1}y_{3,2},\ y_{1,1}y_{1,2}x_{2,1}x_{2,2}y_{3,1}y_{3,2},\ x_{1,1}y_{1,2}y_{2,1}y_{2,2}y_{3,1}y_{3,2},\\
	&\phantom{ =\ \big(.}y_{1,1}y_{1,2}x_{2,1}y_{2,2}y_{3,1}y_{3,2},\ y_{1,1}y_{1,2}y_{2,1}y_{2,2}y_{3,1}y_{3,2}\big).
	\end{align*}
	One can easily verify that $(H_P^2)^\wp=H_{P(2)}$ with respect to our modified polarization. For instance, consider $(x_1y_2y_3)(y_1x_2y_3)\in G(H_P^2)$. Then
	$$
	\big((x_1y_2y_3)(y_1x_2y_3)\big)^\wp\ =\ (x_1y_1x_2y_2y_3^2)^\wp\ =\ x_{1,1}y_{1,2}x_{2,1}y_{2,2}y_{3,1}y_{3,2}\ =\ u_{\mathcal{I}}\in G(H_{P(2)}),
	$$
	where $\mathcal{I}=\{x_{1,1},x_{2,1}\}\in\mathcal{J}(P(2))$.
\end{Example}

\begin{Theorem}\label{Thm:PosetEll}
	Let $(P,\succeq)$ be a finite poset. Then, for any $\ell\ge1$,
	$$
	(H_P^\ell)^\wp\ =\ H_{P(\ell)}.
	$$
\end{Theorem}
\begin{proof}
	By Lemma \ref{lem:Hibi}, $H_P^\ell$ is minimally generated by the monomials
	$$
	u_{\mathcal{I}_1}u_{\mathcal{I}_2}\cdots u_{\mathcal{I}_\ell},\ \ \ \text{with}\ \ \ \mathcal{I}_1\subseteq \mathcal{I}_2\subseteq\dots\subseteq \mathcal{I}_\ell,\ \ \ \mathcal{I}_j\in\mathcal{J}(P),\ \ \ j=1,\dots,\ell,
	$$
	while $H_{P(\ell)}\subset K\big[x_{i,j},y_{i,j}:i\in[n],j\in[\ell]\big]$ is generated by the squarefree monomials
	$$
	u_{\mathcal{I}}\ =\ \big(\prod_{x_{i,j}\in\mathcal{I}}x_{i,j}\big)\big(\prod_{x_{i,j}\in P(\ell)\setminus\mathcal{I}}y_{i,j}\big),\ \ \ \ \ \mathcal{I}\in\mathcal{J}(P(\ell)).
	$$
	Hence, to get the assertion, we must show that for all $u_{\mathcal{I}_1}u_{\mathcal{I}_2}\cdots u_{\mathcal{I}_\ell}\in G(H_P^\ell)$, the monomial $(u_{\mathcal{I}_1}u_{\mathcal{I}_2}\cdots u_{\mathcal{I}_\ell})^\wp$ is equal to $u_{\mathcal{I}}$, for some poset ideal $\mathcal{I}\in\mathcal{J}(P(\ell))$, and conversely, given any $u_{\mathcal{I}}\in G(H_{P(\ell)})$, with $\mathcal{I}\in\mathcal{J}(P(\ell))$, then there exist $\mathcal{I}_1\subseteq \mathcal{I}_2\subseteq\dots\subseteq \mathcal{I}_\ell$, $\mathcal{I}_i\in\mathcal{J}(P)$, $i=1,\dots,\ell$, such that $u_{\mathcal{I}}=(u_{\mathcal{I}_1}u_{\mathcal{I}_2}\cdots u_{\mathcal{I}_\ell})^\wp$.
	
	Let $u_{\mathcal{I}_1}u_{\mathcal{I}_2}\cdots u_{\mathcal{I}_\ell}\in G(H_P^\ell)$, with $\mathcal{I}_1\subseteq \mathcal{I}_2\subseteq\dots\subseteq \mathcal{I}_\ell$. Set $k_i=\big|\{r\ :\ x_i\in\mathcal{I}_r\}\big|$ and $m_i=\big|\{s\ :\ x_i\notin\mathcal{I}_s\}\big|$. Therefore, $k_i+m_i=\ell$ and
	$$
	(u_{\mathcal{I}_1}u_{\mathcal{I}_2}\cdots u_{\mathcal{I}_\ell})^\wp\ =\ \big(\prod_{x_i\in P}\prod_{r=1}^{k_i}x_{i,r}\big)\big(\prod_{x_i\in P}\prod_{s=k_i+1}^{\ell}y_{i,s}\big).
	$$
	We claim that $\mathcal{I}=\big\{x_{i,r}\ :\ r=1,\dots,k_i\big\}$ is a poset ideal of $P(\ell)$. From this, it will follow that $(u_{\mathcal{I}_1}u_{\mathcal{I}_2}\cdots u_{\mathcal{I}_\ell})^\wp=u_{\mathcal{I}}\in G(H_{P(\ell)})$, as wanted. Indeed, let $x_{i,r}\in\mathcal{I}$ and $x_{j,s}\in P(\ell)$ such that $x_{j,s}\preceq_\ell x_{i,r}$. We must prove that $x_{j,s}\in\mathcal{I}$. By definition of the order $\succeq_\ell$ we have $x_j\preceq x_i$ and $r\ge s$. If $j=i$, then also $x_{i,s}\in\mathcal{I}$,  by the polarization technique. Let $j\ne i$. We note that $k_j\ge k_i$. Indeed, for any $c$ such that $x_i\in\mathcal{I}_c$, one has that $x_j\in\mathcal{I}_c$, since $x_j\preceq x_i$. Thus, if $x_{i,r}$ divides $(u_{\mathcal{I}_1}u_{\mathcal{I}_2}\cdots u_{\mathcal{I}_\ell})^\wp$, then $x_{j,r}$ divides $(u_{\mathcal{I}_1}u_{\mathcal{I}_2}\cdots u_{\mathcal{I}_\ell})^\wp$. For any $d\le r$, $x_{j,d}$ divides $(u_{\mathcal{I}_1}u_{\mathcal{I}_2}\cdots u_{\mathcal{I}_\ell})^\wp$, by the polarization technique. Since $s\le r$, then $x_{j,s}$ divides $(u_{\mathcal{I}_1}u_{\mathcal{I}_2}\cdots u_{\mathcal{I}_\ell})^\wp$, as wanted.
	
	Conversely, let $\mathcal{I}\in\mathcal{J}(P(\ell))$ be a poset ideal and $u_{\mathcal{I}}\in G(H_{P(\ell)})$. Set
	$$
	\mathcal{I}_k=\big\{x_i\ :\ x_{i,\ell+1-k}\in\mathcal{I}\big\},\ \ \ \text{for}\ \ \ k=1,\dots,\ell.
	$$
	We claim that the sets $\mathcal{I}_j$ are poset ideals of $P$ and that $\mathcal{I}_1\subseteq \mathcal{I}_2\subseteq\dots\subseteq \mathcal{I}_\ell$. From this, it will follow that $u_{\mathcal{I}}=(u_{\mathcal{I}_1}u_{\mathcal{I}_2}\cdots u_{\mathcal{I}_\ell})^\wp\in G((H_P^\ell)^\wp)$, as wanted.
	
	 Fix $k\in[\ell]$ and let $x_i\in\mathcal{I}_k$ and $x_j\preceq x_i$. We must prove that $x_j\in\mathcal{I}_k$, too. Since $x_j\preceq x_i$, then $x_{j,\ell+1-k}\preceq_\ell x_{i,\ell+1-k}$ by definition of the order $\succeq_\ell$. Thus $x_{j,\ell+1-k}\in\mathcal{I}$ and consequently $x_j\in\mathcal{I}_k$, as wanted.

	Now, let $k\in[\ell]$ with $k<\ell$. We prove that $\mathcal{I}_{k}\subseteq\mathcal{I}_{k+1}$. As a consequence, it follows that $\mathcal{I}_1\subseteq \mathcal{I}_2\subseteq\dots\subseteq \mathcal{I}_\ell$, as desired. Let $x_{i}\in\mathcal{I}_k$, then $x_{i,\ell+1-k}\in\mathcal{I}$. Since $\ell+1-k>\ell+1-(k+1)$, we get that $x_{i,\ell+1-k}\succeq_\ell x_{i,\ell+1-(k+1)}$ and so $x_{i,\ell+1-(k+1)}\in\mathcal{I}$, too. Thus $x_i\in\mathcal{I}_{k+1}$. The assertion follows.
	\end{proof}

The following key lemma proved by Olteanu \cite[Proposition 5.3]{Olteanu2009} follows by a more general result stated by Jahan \cite[Lemma 3.3]{Jahan2007}.
\begin{Lemma}\label{Lem:IwpSets}
	Let $I\subset S$ be a monomial ideal. Then $I$ has linear quotients with admissible order $u_1>u_2>\dots>u_m$ of $G(I)$ if and only if $I^\wp$ has linear quotients with admissible order $u_1^\wp>u_2^\wp>\dots>u_m^\wp$ of $G(I^\wp)$.
\end{Lemma}

The previous lemma obviously also holds with respect to our modified polarization. Finally, we obtain the next result which gives a positive answer to Conjecture \ref{Conj:PowersHSCMverywell} for the class of bipartite graphs. 
\begin{Corollary}\label{Cor:HSHPPowers}
	Let $G$ be a Cohen--Macaulay bipartite graph with $2n$ vertices. Then $\HS_{k}((I(G)^\vee)^\ell)$ has linear quotients with respect to the lexicographic order induced by $x_n>y_n>x_{n-1}>y_{n-1}>\dots>x_1>y_1$, for all $k\ge0$ and all $\ell\ge1$.
	\end{Corollary}
\begin{proof}
	By \cite[Theorem 9.1.3]{JT}, there exists a poset $(P,\succeq)$ such that $I(G)^\vee=H_P$. Let $\ell\ge1$. Then by Theorem \ref{Thm:PosetEll}, $(H_P^\ell)^\wp=H_{P(\ell)}$. By Theorem \ref{Thm:HSCMVWCG}, for all $k\ge0$, $\HS_k(H_{P(\ell)})$ has linear quotients with respect to the lexicographic order induced by
	$$
	x_{n,\ell}>y_{n,\ell}>x_{n,\ell-1}>y_{n,\ell-1}>\dots> x_{n,1}>y_{n,1}>x_{n-1,\ell}>y_{n-1,\ell}>\dots>x_{1,1}>y_{1,1}.
	$$
	By Lemma \ref{Lem:HSpolarization}, $\HS_{k}(H_{P(\ell)})=\HS_{k}((H_P^\ell)^\wp)=\HS_k(H_P^\ell)^\wp$. Finally, applying Lemma \ref{Lem:IwpSets}, we obtain that $\HS_k(H_P^\ell)=\HS_k((I(G)^\vee)^\ell)$ has linear quotients with respect to the lexicographic order induced by $x_n>y_n>x_{n-1}>y_{n-1}>\dots>x_1>y_1$, as wanted.
\end{proof}

\textit{Acknowledgment.} We thank the referees for their helpful suggestions, that allowed
us to improve the quality of the paper.

\end{document}